\documentclass[11pt]{amsart}   

\usepackage{verbatim}
\usepackage{amssymb,amsthm,amsmath}
\usepackage{color}
\definecolor{darkblue}{rgb}{0, 0, .4}
\definecolor{grey}{rgb}{.7, .7, .7}

\usepackage{ifpdf}
\ifpdf
  \usepackage[pdftex]{epsfig}

  \usepackage{lscape}

  \usepackage[breaklinks]{hyperref}
  \hypersetup{
            colorlinks=true,
            linkcolor=darkblue,
            anchorcolor=darkblue,
            citecolor=darkblue,
            pagecolor=darkblue,
            urlcolor=darkblue,
            pdftitle={},
            pdfauthor={}
  }
\else
  \usepackage[dvips]{epsfig}
  \usepackage{lscape}

  \newcommand{\href}[2]{#2}
  \newcommand{\url}[2]{#2}
\fi

\newtheorem{theorem}{Theorem}[section]
\newtheorem{lemma}[theorem]{Lemma}

\theoremstyle{definition}
\newtheorem{definition}[theorem]{Definition}
\newtheorem{example}[theorem]{Example}

\theoremstyle{remark}
\newtheorem{remark}[theorem]{Remark}

\numberwithin{equation}{section}

\theoremstyle{theorem}
\newtheorem{corollary}[theorem]{Corollary}

\newtheorem{proposition}[theorem]{Proposition}

\newcommand{\N}[0]{\mathbb{N}}
\newcommand{\Z}[0]{\mathbb{Z}}

\newcommand{\row}{\rightarrow}

\newcommand{\ontop}[2]{\genfrac{}{}{0pt}{}{#1}{#2}}

\input xy 
\xyoption{all} 

\usepackage{graphicx}

\newcommand{\tld}{\widetilde}
\newcommand{\lam}{\lambda}

\newcommand{\res}{\mathsf{res}}
\newcommand{\mres}{\mathsf{mres}}
\newcommand{\p}{\|}

\newcommand{\Fa}{\mathcal{A}}
\newcommand{\Fc}{\mathcal{C}}
\newcommand{\Fb}{\mathcal{B}}
\newcommand{\Fw}{\mathcal{W}}
\newcommand{\Fr}{\mathcal{R}}

\newcommand{\abacus}[1]{{ \tiny \xymatrix @-2.2pc { #1 } }}

\newcommand{\ci}[1]{{\xy*{ #1 }*\cir<10pt>{}\endxy}}  
\newcommand{\nc}[1]{{\xy*{ #1 }*i\cir<10pt>{}\endxy}}  

\usepackage{tikz}

\def\myscalebox{ \begingroup \catcode`\&=\active \myscaleboxi}
\def\myscaleboxi#1#2{ \scalebox{#1}{#2} \endgroup} 

\newcommand\tightpartition[1]{
\begin{tikzpicture}
\tikzstyle{shaded}=[rectangle, minimum size=1.4cm, fill=gray!20, draw=gray!20];
\tikzstyle{blueshaded}=[rectangle, minimum size=1.4cm, fill=blue!15, draw=blue!15];
\tikzstyle{unshaded}=[rectangle, minimum size=1.4cm];
\tikzstyle{part}=[rectangle, minimum size=1.4cm, draw=black];
\tikzstyle{shadedpart}=[rectangle, minimum size=1.4cm, fill=gray!20, draw=black];
\tikzstyle{blueshadedpart}=[rectangle, minimum size=1.4cm, fill=blue!15, draw=black];
\matrix[column sep=-0.35cm,row sep=-0.35cm]
{ #1 };
\end{tikzpicture} } 

\newcommand\partition[1]{
\begin{tikzpicture}
\tikzstyle{shaded}=[rectangle, minimum size=1.4cm, fill=gray!20, draw=gray!20];
\tikzstyle{blueshaded}=[rectangle, minimum size=1.4cm, fill=blue!15, draw=blue!15];
\tikzstyle{unshaded}=[rectangle, minimum size=1.4cm];
\tikzstyle{part}=[rectangle, minimum size=1.4cm, draw=black];
\tikzstyle{shadedpart}=[rectangle, minimum size=1.4cm, fill=gray!20, draw=black];
\tikzstyle{blueshadedpart}=[rectangle, minimum size=1.4cm, fill=blue!15, draw=black];
\matrix[column sep=0pt,row sep=0pt ]
{ #1 };
\end{tikzpicture} } 

\newcommand\ns[1]{ \node [shaded]{$#1$}; }
\newcommand\nn[1]{ \node [unshaded]{$#1$}; }
\newcommand\np[1]{ \node [part]{$#1$}; }
\newcommand\nsp[1]{ \node [shadedpart]{$#1$}; }
\newcommand\nb[1]{ \node [blueshaded]{$#1$}; }
\newcommand\nbp[1]{ \node [blueshadedpart]{$#1$}; }

\begin{document}

\title{Abacus models for parabolic quotients of affine Weyl groups}

\begin{abstract}
We introduce abacus diagrams that describe the minimal length coset
representatives of affine Weyl groups in types $\widetilde{C}/C$,
$\widetilde{B}/D$, $\widetilde{B}/B$ and $\widetilde{D}/D$.  These abacus
diagrams use a realization of the affine Weyl group $\widetilde{C}$ due to
Eriksson to generalize a construction of James for the symmetric group.  We
also describe several combinatorial models for these parabolic quotients that
generalize classical results in type $\widetilde{A}$ related to core
partitions.
\end{abstract}

\author{Christopher R.\ H.\ Hanusa}
\address{Department of Mathematics \\ Queens College (CUNY) \\ 65-30 Kissena Blvd. \\ Flushing, NY 11367}
\email{\href{mailto:christopher.hanusa@qc.cuny.edu}{\texttt{christopher.hanusa@qc.cuny.edu}}}
\urladdr{\url{http://people.qc.cuny.edu/faculty/christopher.hanusa/}}

\author{Brant C. Jones}
\address{Department of Mathematics and Statistics, MSC 1911, James Madison University, Harrisonburg, VA 22807}
\email{\href{mailto:brant@math.jmu.edu}{\texttt{brant@math.jmu.edu}}}
\urladdr{\url{http://www.math.jmu.edu/\~brant/}}


\date{\today}

\maketitle


\bigskip
\section{Introduction}\label{s:background}

Let $\widetilde{W}$ be an affine Weyl group, and $W$ be the corresponding finite
Weyl group.  Then the cosets of $W$ in $\widetilde{W}$, often denoted
$\widetilde{W}/W$, have a remarkable combinatorial structure with connections to
diverse structures in algebra and geometry, including affine Grassmannians
\cite{LM,billey--mitchell}, characters and modular representations for the
symmetric group \cite{JK,kleshchev,G}, and crystal bases of quantum groups
\cite{MM,Kwon}.  Combinatorially, the elements in $\widetilde{W}$ can be
understood as pairs from $\widetilde{W}/W \times W$ by the parabolic
decomposition (see e.g. \cite[Proposition 2.4.4]{b-b}).

In type $\widetilde{A}$, these cosets correspond to a profoundly versatile
combinatorial object known as an abacus diagram.  From the abacus diagram, one
can read off related combinatorial objects such as root lattice coordinates,
core partitions, and the bounded partitions used in \cite{Erik2}, \cite{LM} and
\cite{billey--mitchell}.  The goal of the present paper is to extend the abacus
model to types $\widetilde{B}$, $\widetilde{C}$, and $\widetilde{D}$, and
define analogous families of combinatorial objects in these settings.  Some of
these structures are not, strictly speaking, new.  Nevertheless, we believe
that our development using abacus diagrams unifies much of the folklore, and we
hope that it will be useful to researchers and students interested in extending
results from type $\widetilde{A}$ to the other affine Weyl groups.  In this
sense, our paper is a companion to \cite{berg-jones-vazirani}, \cite{Erik2} and
\cite{lascoux-cores}.

The following diagram illustrates six families of combinatorial objects that are
all in bijection.  Each family has an action of $\widetilde{W}$, a Coxeter
length function, and is partially ordered by the Bruhat order.  We will devote
one section to each of these objects and give the bijections between them using
type-independent language.
\[ \myscalebox{0.7}{\begin{tikzpicture}
    \matrix[row sep=1.0cm,column sep=3.0cm] {
    \node[rectangle,draw] (A) {\parbox{1.2in}{sorted mirrored \newline permutations of $\mathbb{Z}$}}; &
    \node[rectangle,draw] (B) {\parbox{1.2in}{Abacus diagrams}}; &
    \node[rectangle,draw] (C) {\parbox{1.3in}{Core partitions}}; &
    \node[rectangle,draw] (D) {\parbox{1.2in}{Bounded partitions}}; \\
    &
    \node[rectangle,draw] (E) {\parbox{1.2in}{Root lattice points}}; &
    \node[rectangle,draw] (F) {\parbox{1.3in}{Canonical reduced expressions for minimal length coset representatives}}; &
    & \\
    };
    \path[->]
        (A) edge[thick] (B)
        (B) edge[thick] (C)
        (C) edge[thick] (D)
        (B) edge[thick] (E)
        (C) edge[thick] (F)
        ;
\end{tikzpicture}} \]
We believe that the abacus diagrams and core partitions we introduce have not
appeared in this generality before.  In fact, we show in
Theorem~\ref{t:core_bruhat} that our construction answers a question of Billey
and Mitchell; see Section~\ref{s:bruhat}.  Several authors have used
combinatorics related to bounded partitions, and we show how these objects are
naturally related to abaci in Section~\ref{s:bounded}.  We also obtain some
formulas for Coxeter length in Section~\ref{s:coxeter_length_formulas} that
appear to be new.  To avoid interrupting the exposition, we postpone a few of
the longer arguments from earlier sections to Section~\ref{s:proofs}.
Section~\ref{s:future_work} briefly suggests some ideas for further research.

In order to simplify the notation, we use a convention of overloading the
definitions of our bijections.  We let the output of the functions ($\Fw$,
$\Fa$, $\Fc$, $\Fr$, $\Fb$) be the corresponding combinatorial interpretation
($\mathbb{Z}$-permutation, abacus diagram, core partition, canonical reduced
expression, bounded partition, respectively), no matter the input.  For
example, if $\beta$ is a bounded partition, then its corresponding abacus
diagram is $\Fa(\beta)$.

\bigskip
\section{George groups and $\mathbb{Z}$-permutations}

\subsection{Definitions}

We follow the conventions of Bj\"orner and Brenti in \cite[Chapter 8]{b-b}.

\begin{definition}
Fix a positive integer $n$ and let $N=2n+1$.  We say that a bijection $w:\Z\row \Z$ is a
{\bf mirrored $\mathbb{Z}$-permutation} if
\begin{equation}\label{e:8.42}
w(i+N) = w(i) + N, \text{ and }
\end{equation}
\begin{equation}\label{e:8.43}
w(-i) = -w(i).
\end{equation}
for all $i\in \Z$.
\end{definition}

Eriksson and Eriksson \cite{Erik2} use these mirrored permutations to give a
unified description of the finite and affine Weyl groups, based on ideas from
\cite{HErik94}.  It turns out that the collection of mirrored
$\mathbb{Z}$-permutations forms a realization of the affine Coxeter group
$\widetilde{C}_n$, where the group operation is composition of
$\mathbb{Z}$-permutations.  Since the Coxeter groups $\widetilde{B}_n$ and
$\widetilde{D}_n$ are subgroups of $\widetilde{C}_n$, every element in any of
these groups can be represented as such a permutation.  Green
\cite{green-full-heaps} uses the theory of full heaps to obtain this and
related representations of affine Weyl groups.

\begin{remark}
A mirrored $\mathbb{Z}$-permutation $w$ is completely determined by its action
on $\{1, 2, \ldots n\}$.  Also, Equations~\eqref{e:8.42} and \eqref{e:8.43}
imply that $w(i)=i$ for all $i=0$ mod $N$.  
\end{remark}

We have Coxeter generators whose images $\left(w(1), w(2), \ldots, w(n)\right)$
are given by
\[ s_i = \left(1, 2, \ldots, i-1, i+1, i, i+2, \ldots, n \right) \ \ \ \ \text{for $1 \leq i \leq n-1$} \]
\[ s_0^C = \left(-1, 2, 3, \ldots, n \right) \]
\[ s_0^D = \left(-2, -1, 3, 4, \ldots, n \right) \]
\[ s_n^C = \left(1, 2, \ldots, n-1, n+1 \right) \]
\[ s_n^D = \left(1, 2, \ldots, n-2, n+1, n+2 \right) \]
and we extend each of these to an action on $\mathbb{Z}$ via $\eqref{e:8.42}$
and $\eqref{e:8.43}$.  Observe that each of these generators interchange
infinitely many entries of $\mathbb{Z}$ by Equation~(\ref{e:8.42}).

\begin{theorem}\label{t:bb8}
The collection of mirrored $\mathbb{Z}$-permutations that satisfy the conditions
in the second column of Table~\ref{t:mzp_def} form a realization of the
corresponding affine Coxeter group $\widetilde{C}$, $\widetilde{B}$, or
$\widetilde{D}$.  The collection of mirrored $\mathbb{Z}$-permutations that
additionally satisfy the sorting conditions in the third column of
Table~\ref{t:mzp_def} form a collection of minimal length coset representatives
for the corresponding parabolic quotient shown in the first column of
Table~\ref{t:mzp_def}.  The corresponding Coxeter graph is shown in the fourth
column of Table~\ref{t:mzp_def}.
\end{theorem}
\begin{proof}
Proofs can be found in \cite{Erik2} and \cite[Section 8]{b-b}.
\end{proof}

To describe the essential data that determines an mirrored
$\mathbb{Z}$-permutation, we observe an equivalent symmetry. 


\begin{landscape}
\begin{table}
\begin{tabular}{|p{0.3in}|p{2.5in}|p{2.5in}|p{2.5in}|}
\hline
Type & Conditions on $\mathbb{Z}$-permutation for Coxeter group elements & Sorting conditions for minimal
length coset representatives & Coxeter graph \\
\hline
$\widetilde{C}/C$ & \  & $w(1) < w(2) < \cdots < w(n) < w(n+1)$ 
& 
\myscalebox{0.45}{\begin{tikzpicture}
    \matrix[row sep=1.0cm,column sep=1.0cm] {
    \node[circle,draw] (A) {$s_0^C$}; & 
    \node[circle,draw] (B) {$s_1$}; &
    \node[circle,draw] (C) {$s_2$}; &
    \node (D) {$\cdots$}; &
    \node[circle,draw] (E) {$s_{n-2}$}; &
    \node[circle,draw] (F) {$s_{n-1}$}; &
    \node[circle,draw] (G) {$s_{n}^C$}; \\
    };
    \path[-]
        (A) edge[thick] node[above,midway]{4} (B) 
        (B) edge[thick] (C) 
        (C) edge[thick] (D) 
        (D) edge[thick] (E) 
        (E) edge[thick] (F) 
        (F) edge[thick] node[above,midway]{4} (G) ;
\end{tikzpicture}}
\\
\hline
$\widetilde{B}/B$ & $\left|i \in \Z : i \leq 0, w(i) \geq 1\right| \equiv 0 \mod 2$

\vspace{0.2in}

{\small (By (\ref{e:8.43}), this is equivalent to requiring the number of negative
entries lying to the right of position zero is even.)}
& 
$w(1) < w(2) < \cdots < w(n) < w(n+1)$

\vspace{0.2in}

{\small (Comparing with the condition on the previous row, we see that elements
of $\widetilde{B}_n/B_n$ are elements of $\widetilde{C}_n/C_n$.) }
& 
\myscalebox{0.5}{\begin{tikzpicture}
    \matrix[row sep=1.0cm,column sep=1.0cm] {
    \node[circle,draw] (B) {$s_{1}$}; &
    & & & & \\
    \node[circle,draw] (A) {$s_0^D$}; & 
    \node[circle,draw] (C) {$s_2$}; &
    \node (D) {$\cdots$}; &
    \node[circle,draw] (E) {$s_{n-2}$}; &
    \node[circle,draw] (F) {$s_{n-1}$}; &
    \node[circle,draw] (G) {$s_{n}^C$}; \\
    };
    \path[-]
        (A) edge[thick] (C)
        (B) edge[thick] (C) 
        (C) edge[thick] (D) 
        (D) edge[thick] (E) 
        (E) edge[thick] (F) 
        (F) edge[thick] node[above,midway]{4} (G) ;
\end{tikzpicture}}
\\
\hline
$\widetilde{B}/D$ & 
$\left|i \in \Z : i \leq n, w(i) \geq n+1\right| \equiv 0 \mod 2$
& 
$w(1) < w(2) < \cdots < w(n) < w(n+2)$

\vspace{0.2in}

{\small (Comparing with the previous conditions, we see that elements of
$\widetilde{B}_n/D_n$ are {\em not} necessarily elements of
$\widetilde{C}_n/C_n$, even though $\widetilde{B}_n$ is a subgroup of
$\widetilde{C}_n$.)}

& 
\myscalebox{0.5}{\begin{tikzpicture}
    \matrix[row sep=1.0cm,column sep=1.0cm] {
    & & & & &
    \node[circle,draw] (F) {$s_{n-1}$}; \\
    \node[circle,draw] (A) {$s_0^C$}; & 
    \node[circle,draw] (B) {$s_1$}; &
    \node[circle,draw] (C) {$s_2$}; &
    \node (D) {$\cdots$}; &
    \node[circle,draw] (E) {$s_{n-2}$}; &
    \node[circle,draw] (G) {$s_{n}^D$}; \\
    };
    \path[-]
        (A) edge[thick] node[above,midway]{4} (B) 
        (B) edge[thick] (C) 
        (C) edge[thick] (D) 
        (D) edge[thick] (E) 
        (E) edge[thick] (F) 
        (E) edge[thick] (G) ;
\end{tikzpicture}}
\\
\hline
$\widetilde{D}/D$ & 
$\left|i \in \Z : i \leq 0, w(i) \geq 1\right| \equiv 0 \mod 2$ and
$\left|i \in \Z : i \leq n, w(i) \geq n+1\right| \equiv 0 \mod 2$

& 
$w(1) < w(2) < \cdots < w(n) < w(n+2)$

\vspace{0.2in}

{\small (Comparing with the previous conditions, we see that elements of
$\widetilde{D}_n/D_n$ are also elements of $\widetilde{B}_n/D_n$.)}

& 
\myscalebox{0.5}{\begin{tikzpicture}
    \matrix[row sep=1.0cm,column sep=1.0cm] {
    \node[circle,draw] (B) {$s_{1}$}; &
    & & & 
    \node[circle,draw] (F) {$s_{n-1}$}; \\
    \node[circle,draw] (A) {$s_0^D$}; & 
    \node[circle,draw] (C) {$s_2$}; &
    \node (D) {$\cdots$}; &
    \node[circle,draw] (E) {$s_{n-2}$}; &
    \node[circle,draw] (G) {$s_{n}^D$}; \\
    };
    \path[-]
        (A) edge[thick] (C)
        (B) edge[thick] (C) 
        (C) edge[thick] (D) 
        (D) edge[thick] (E) 
        (E) edge[thick] (F) 
        (E) edge[thick] (G) ;
\end{tikzpicture}}
\\
\hline
\end{tabular}
\caption{Realizations of affine Coxeter groups}\label{t:mzp_def}
\end{table}
\end{landscape}


\begin{lemma}{\bf (Balance Lemma)}\label{l:ol_balance}
If $w$ is an mirrored $\mathbb{Z}$-permutation, then we have
\begin{equation}\label{e:local_balance}
w(i) + w(N-i) = N \text{ for all } i = 1, 2, \ldots n.
\end{equation}
Conversely, if $w : \Z \rightarrow \Z$ is a bijection that satisfies (\ref{e:8.42}) and
(\ref{e:local_balance}), then $w$ is an mirrored $\mathbb{Z}$-permutation.
\end{lemma}
\begin{proof}
Equations~\eqref{e:8.42} and \eqref{e:8.43} imply that for all $i\in \Z$,
$w(i)+w(N-i) = w(i)-w(i-N)=w(i)-\big(w(i)-N\big)=N$; in particular, this is true
for $1\leq i\leq n$.  

To prove the converse, we must show that an infinite permutation $w:\Z\row\Z$
satisfying \eqref{e:8.42} and \eqref{e:local_balance} satisfies \eqref{e:8.43}.
Equation~\eqref{e:8.42} implies
\[ w(-i - kN) = w(N-i) - (k+1) N \]
for any $i \in \{1, \ldots, N\}$ and $k \geq 0$.  By \eqref{e:local_balance},
we have $w(N-i) = N - w(i)$ so
\[ w(-i - kN) = (N - w(i)) - (k+1) N = -w(i) - kN = -(w(i) + kN) = -w(i+kN), \]
as was to be shown.
\end{proof}

Given a mirrored $\mathbb{Z}$-permutation, we call the ordered sequence $\left[w(1),
w(2), \ldots, w(2n)\right]$ the {\bf base window} of $w$.  Since the set of
mirrored $\mathbb{Z}$-permutations acts on itself by composition of functions,
we have an action of $\widetilde{W}$ on the base window notation.  The left
action interchanges values while the right action interchanges positions.  We
have labeled the node of the Coxeter graph of $\widetilde{W}$ that is added to
the Coxeter graph of $W$ by $s_0$.  Then, our cosets have the form
$\widetilde{w}W$ (where $\widetilde{w} \in \widetilde{W}$), and the minimal
length coset representatives all have $s_0$ as a unique right descent.  

We can characterize the base windows that arise.

\begin{lemma}\label{l:olwd}
An ordered collection $\left[w(1), w(2), \ldots, w(2n)\right]$ of integers is the base
window for an element of $\widetilde{C}_n$ if and only if
\begin{itemize}
    \item $w(1), \ldots, w(2n)$ have distinct residue mod $N$,
    \item $w(1), \ldots, w(2n)$ are not equivalent to $0$ mod $N$, and
    \item $w(i) + w(N-i) = N$ for each $i = 1, \ldots, 2n = N-1$.
\end{itemize}
\end{lemma}
\begin{proof}
Given such a collection of integers, extend $\left[w(1), \ldots, w(2n)\right]$
to a $\mathbb{Z}$-permutation $w$ using \eqref{e:8.42}, and set $w(iN) = iN$ for
all $i \in \mathbb{Z}$.  The third condition on $\left[w(1), \ldots,
w(2n)\right]$ ensures that $w$ is a mirrored $\mathbb{Z}$-permutation by
the Balance Lemma~\ref{l:ol_balance}.

On the other hand, each of the three conditions is preserved when we apply a
Coxeter generator $s_i$, so each element of $\widetilde{C}_n$ satisfies these
conditions by induction on Coxeter length.
\end{proof}

With the conventions we have adopted in Table~\ref{t:mzp_def}, we can also
prove that no two minimal length coset representatives contain the same entries
in their base window.

\begin{lemma}\label{l:bdoldet}
Suppose $a_1, \ldots, a_n$ is a collection of integers such that each $a_i$ is
equivalent to $i$ mod $N$.  Then, there exists a unique element $w \in
\widetilde{B}_n/D_n$ that contains $a_1, \ldots, a_n$ among the entries of its
base window $\{w(1), \ldots, w(2n)\}$.
\end{lemma}
\begin{proof}
It follows from the Balance Lemma~\ref{l:ol_balance} that whenever $a_i$
appears among the entries of the base window of $w \in \widetilde{B}_n$, then
$N-a_i$ also appears among the entries of the base window of $w$.  Hence, the
entries of the base window $\{a_1, \ldots, a_n, N-a_1, \ldots, N-a_n\}$ are
completely determined by the $a_i$.

For $w$ to be a minimal length coset representative, we must order these
entries to satisfy the condition shown in the third column of Table~\ref{t:mzp_def}
while maintaining the condition shown in the second column of Table~\ref{t:mzp_def}.
Let $\tilde{a}_1, \tilde{a}_2, \ldots, \tilde{a}_{2n}$ denote the entries
$\{a_1, \ldots, a_n, N-a_1, \ldots, N-a_n\}$ of the base window arranged into
increasing order so $\tilde{a}_1 < \tilde{a}_2 < \cdots < \tilde{a}_{2n}$.

By the condition shown in the third column of Table~\ref{t:mzp_def} and \eqref{e:8.42},
we have that $w(n)$ is the only possible descent among the entries of the base
window of $w$, so we have that $w(n+1) < w(n+2)$.  Also, since $w(n) < w(n+2)$,
we have $N - w(n) > N - w(n+2)$ which implies $w(n+1) > w(n-1)$ by
\eqref{e:8.42}.  Since $w(1) < w(2) < \cdots < w(n) < w(n+2) < w(n+3) < \cdots <
w(2n)$ forms an increasing subsequence of length $2n-1$, we have $w(n+2) \leq
\tilde{a}_{n+2}$ and $w(n-1) \geq \tilde{a}_{n-1}$.  Putting these together, we
find 
\[ \tilde{a}_{n-1} \leq w(n-1) < w(n+1) < w(n+2) \leq \tilde{a}_{n+2}. \]  
Thus, $w(n+1)$ must be $\tilde{a}_n$ or $\tilde{a}_{n+1}$.  Therefore, the
entries of the base window of $w \in \widetilde{B}_n/D_n$ are either
$\tilde{a}_1, \tilde{a}_2, \cdots, \tilde{a}_{n-1}, \tilde{a}_n,
\tilde{a}_{n+1}, \tilde{a}_{n+2}, \cdots, \tilde{a}_{2n}$ or $\tilde{a}_1,
\tilde{a}_2, \cdots, \tilde{a}_{n-1}, \tilde{a}_{n+1}, \tilde{a}_{n},
\tilde{a}_{n+2}, \cdots, \tilde{a}_{2n}$.  Since $\tilde{a}_n + \tilde{a}_{n+1}
= N$, precisely one of these satisfies the condition shown in the second column of
Table~\ref{t:mzp_def}.
\end{proof}

\begin{corollary}\label{c:unbw}
Suppose $w, w' \in \widetilde{W}/W$.  If $w \neq w'$ then 
\[ \{w(1), w(2), \ldots, w(2n)\} \neq \{w'(1), w'(2), \ldots, w'(2n)\} \]
as unordered sets.
\end{corollary}
\begin{proof}
This follows from the observation that the sorting conditions completely
determine the ordering of elements in the base window.  Considering the third
column of Table~\ref{t:mzp_def}, this is clear for $\widetilde{C}/C$ and
$\widetilde{B}/B$.  By Lemma~\ref{l:bdoldet}, we see that this holds for
$\widetilde{B}/D$ as well.  Since $\widetilde{D}_n/D_n \subset
\widetilde{B}_n/D_n$, it holds for $\widetilde{D}/D$.
\end{proof}

\bigskip
\section{Abacus diagrams}

\subsection{Definitions}

We now combinatorialize the set of integers that can appear in the base window
of a mirrored $\mathbb{Z}$-permutation as an abacus diagram.  These diagrams
enforce precisely the conditions from Lemma~\ref{l:olwd}.

\begin{definition}
An {\bf abacus diagram} (or simply {\bf abacus}) is a diagram containing $2n$
columns labeled $1, 2, \ldots, 2n$, called {\bf runners}.  Runner $i$ contains
entries labeled by the integers $m N + i$ for each {\bf level} $m$ where
$-\infty < m < \infty$.

We draw the abacus so that each runner is vertical, oriented with $-\infty$ at
the top and $\infty$ at the bottom, with runner $1$ in the leftmost position,
increasing to runner $2n$ in the rightmost position.  Entries in the abacus
diagram may be circled; such circled elements are called {\bf beads}.  Entries
that are not circled are called {\bf gaps}.  The linear ordering of the entries
given by the labels $m N + i$ (for level $m \in \Z$ and runner $1 \leq i \leq
2n$) is called the {\bf reading order} of the abacus which corresponds to
scanning left to right, top to bottom.  (Observe that there are no entries in
the abacus having labels $\{m N : m \in \mathbb{Z}\}$.)

We say that a bead $b$ is {\bf active} if there exist gaps (on any runner) that
occur prior to $b$ in reading order.  Otherwise, we say that the bead is {\bf
inactive}.  A runner is called {\bf flush} if no bead on the runner is preceded
in reading order by a gap on that same runner.  We say that an abacus is {\bf
flush} if every runner is flush.  We say that an abacus is {\bf balanced} if 
\begin{itemize}
    \item there is at least one bead on every runner $i$ for $1 \leq i \leq 2n$, and 
    \item the sum of the labels of the lowest beads on runners $i$ and $N-i$ is
        $N$ for all $i = 1, 2, \ldots, 2n$.
\end{itemize}

We say that an abacus is {\bf even} if there exists an even number of gaps
preceding $N$ in reading order.
\end{definition}

\begin{definition}\label{def:abacus}
Given a mirrored $\mathbb{Z}$-permutation $w$, we define $\Fa(w)$ to be the
flush abacus whose lowest bead in each runner is an element of $\{ w(1),
w(2), \ldots, w(2n) \}$.
\end{definition}

Note that this is well-defined by Lemma~\ref{l:olwd}.  Also, $\Fa(w)$ is always
balanced by Lemma~\ref{l:ol_balance}, so the level of the lowest bead on runner
$i$ is the negative of the level of the lowest bead on runner $N-i$.  In the
rest of the paper, we will implicitly assume that all abaci are balanced and
flush unless otherwise noted.

\begin{example}
For the minimal length coset representative $w \in \tld{C}_3/C_3$ whose base
window is \\ $[-11,-9,-1,8,16,18]$, the balanced flush abacus $a=\Fa(w)$ is given in
Figure~\ref{fig:abacus}.  
\begin{figure}[!h]
\epsfig{figure=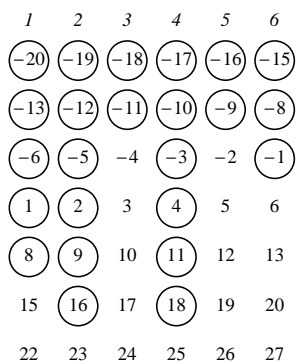, height=2in}
\caption{The balanced flush abacus diagram $a=\Fa(w)$ for the minimal length coset representative $w=[-11,-9,-1,8,16,18] \in \tld{C}_3/C_3$.  The circled entries are the beads; the noncircled entries are the gaps.}
\label{fig:abacus}
\end{figure}
\label{example}
\end{example}

We record some structural facts about balanced flush abaci to be used later.

\begin{lemma}\label{l:ab_balance}
For each $i \in \mathbb{Z}$, we have that entry $N+i$ is a gap if and
only if entry $N-i$ is a bead.
\end{lemma}
\begin{proof}
This follows from the definition together with the Balance
Lemma~\ref{l:ol_balance}.
\end{proof}

\begin{lemma}\label{l:balance_up_low}
Fix an abacus $a$ and consider a single row $r$ of $a$.  If there exists $i$
such that the entries in columns $i$ and $N-i$ are both beads, then the level
of row $r$ is $\leq 0$.  Similarly, if there exists $i$ such that the entries
in columns $i$ and $N-i$ are both gaps, then the level of row $r$ is $> 0$.  In
particular, we cannot have both of these conditions holding at the same time
for a given row of $a$.
\end{lemma}
\begin{proof}
This follows from the Balance Lemma~\ref{l:ol_balance}.
\end{proof}

\begin{table}
\begin{tabular}{|p{0.3in}|p{2.0in}|}
\hline
Type & Conditions on abaci \\
\hline
$\widetilde{C}/C$ & balanced flush abaci \\
\hline
$\widetilde{B}/B$ & even balanced flush abaci \\
\hline
$\widetilde{B}/D$ & balanced flush abaci \\
\hline
$\widetilde{D}/D$ & even balanced flush abaci \\
\hline
\end{tabular}
\caption{Abaci for $\widetilde{W}/W$}\label{t:to_abaci}
\end{table}

\begin{lemma}\label{l:ab_bij}
For each of $\widetilde{C}/C$, $\widetilde{B}/B$, $\widetilde{B}/D$ and
$\widetilde{D}/D$, the map $\Fa$ is a bijection from $\widetilde{W}/W$ to the
set of abaci shown in column 2 of Table~\ref{t:to_abaci}.
\end{lemma}
\begin{proof}
This follows from Corollary~\ref{c:unbw} and Lemma~\ref{l:olwd} for
$\widetilde{C}/C$.  In types $\widetilde{B}/B$ and $\widetilde{D}/D$, the
condition
\[ \left|i \in \Z : i \leq 0, w(i) \geq 1\right| \equiv 0 \mod 2 \]
is equivalent to the even condition on abaci.  
To see this, recall that the entries in the base window of a mirrored
permutation consist of the labels of the lowest beads in each runner of the
abacus.  Therefore, the positive entries of a mirrored permutation that appear
to the left of the base window correspond to the beads lying directly above
some bead in the abacus that lies to the right of $N$ in reading order.  The
set of beads in the abacus succeeding $N$ in reading order has the same
cardinality as the set of gaps preceding $N$ in reading order by
Lemma~\ref{l:ab_balance}.

As explained in Lemma~\ref{l:bdoldet}, the condition 
\[ \left|i \in \Z : i \leq n, w(i) \geq n+1\right| \equiv 0 \mod 2 \]
in type $\widetilde{B}/D$ only changes the ordering of the sorted entries in
the base window, not the set of entries themselves.

The result then follows from Corollary~\ref{c:unbw}.
\end{proof}

\subsection{Action of $\widetilde{W}$ on the abacus}\label{s:ab_action}

If we translate the action of the Coxeter generators on the mirrored
$\mathbb{Z}$-permutations through the bijection $\Fa$, we find that 
\begin{itemize}
    \item $s_i$ interchanges column $i$ with column $i+1$ and
            interchanges column $2n-i$ with column $2n-i+1$, for $1 \leq i \leq n-1$
    \item $s_0^C$ interchanges column $1$ and $2n$, and then shifts the lowest
        bead on column $1$ down one level towards $\infty$, and shifts the
        lowest bead on column $2n$ up one level towards $-\infty$
    \item $s_0^D$ interchanges columns $1$ and $2$ with columns $2n-1$
        and $2n$, respectively, and then shifts the lowest beads on columns $1$
        and $2$ down one level each towards $\infty$, and shifts the lowest
        beads on columns $2n-1$ and $2n$ up one level each towards $-\infty$
    \item $s_n^C$ interchanges column $n$ with column $n+1$ 
    \item $s_n^D$ interchanges columns $n-1$ and $n$ with columns $n+1$
        and $n+2$, respectively.
\end{itemize}

\bigskip
\section{Root lattice points}

\subsection{Definitions}

Following \cite[Section 4]{humphreys}, let $\{ e_1, e_2, \dots, e_{n} \}$ be an
orthonormal basis of the Euclidean space $V = \mathbb{R}^{n}$ and denote the
corresponding inner product by $\langle\cdot, \cdot\rangle$.  Define the {\bf simple roots}
$\alpha_i$ and the {\bf longest root} $\widetilde{\alpha}$ for each
type ${W}_n \in \{{B}_n, {C}_n, {D}_n\}$ as in \cite[page 42]{humphreys}.
The $\mathbb{Z}$-span $\Lambda_R$ of the simple roots is called the {\bf root
lattice}, and we may identify $V$ with $\mathbb{R} \otimes_{\mathbb{Z}} \Lambda_R$
because the simple roots in types $B_n$, $C_n$ and $D_n$ are linearly independent.  

There is an action of $\widetilde{W}$ on $V$ in which $s_i$ is the reflection
across the hyperplane perpendicular to $\alpha_i$ for $i = 1, 2, \ldots, n$ and
$s_0$ is the affine reflection
\[ s_0(v) = v - ( \langle v,\widetilde{\alpha} \rangle - 1 ) \frac{2}{\langle\widetilde{\alpha},
\widetilde{\alpha} \rangle}\widetilde{\alpha}. \]

Suppose $w$ is a minimal length coset representative in $\widetilde{W}/W$ and
define the {\bf root lattice coordinate} of $w$ to be the result of acting on $0
\in V$ by $w$.  

\begin{theorem}
The root lattice coordinate of an element $w \in \widetilde{W}/W$ is
\[ \sum_{i=1}^n \mathsf{level}_i(\Fa(w)) e_i \] 
where $\mathsf{level}_i(\Fa(w))$ denotes the level of the lowest bead in column
$i$ of the abacus $\Fa(w)$.  Moreover, this is a bijection to the collections
of root lattice coordinates shown in Table~\ref{t:to_rlc}.
\end{theorem}
\begin{proof}
Once we identify the Coxeter graphs from Table~\ref{t:mzp_def} with those in
\cite{humphreys}, it is straightforward to verify that the action of
$\widetilde{W}$ on the root lattice is the same as the action of
$\widetilde{W}$ on the levels of the abacus given in Section~\ref{s:ab_action}.

For example, in $\widetilde{B}_n/B_n$, we have $\alpha_n = e_n$ so
\[ s_{n} (a_1 e_1 + \cdots + a_n e_n) = (a_1 e_1 + \cdots + a_n e_n) - (a_n - 0) \frac{2}{\langle \alpha_n, \alpha_n \rangle} \alpha_n \]
\[ = a_1 e_1 + \cdots + a_{n-1} e_{n-1} - a_n e_n \]
and this corresponds to interchanging columns $n$ and $n+1$ in the abacus by the
Balance Lemma~\ref{l:ol_balance}.  Similarly, reflection through the
hyperplane orthogonal to a root of the form $e_{n-1} + e_n$ corresponds to
interchanging columns $n-1$ and $n$ with columns $n+1$ and $n+2$, respectively.
The generators $s_0^C$ and $s_0^D$ are affine
reflections, so we need to shift the level by 1 as described in
Section~\ref{s:ab_action}.  For example, in $\widetilde{C}_n/C_n$, we have
$\widetilde{\alpha} = 2e_1$ so
\[ s_{0} (a_1 e_1 + \cdots + a_n e_n) = (a_1 e_1 + \cdots + a_n e_n) - (2a_1 - 1) \frac{2}{\langle \widetilde{\alpha}, \widetilde{\alpha} \rangle} \widetilde{\alpha} \]
\[ = (-a_1+1) e_1 + a_2 e_2 \cdots + a_n e_n. \]

Because the number of beads to the right of $N$ in an abacus is $\sum_{i=1}^n |
\mathsf{level}_i(\Fa(w))|$, it follows from the proof of Lemma~\ref{l:ab_bij}
that the correspondence between abaci and root lattice coordinates is a
bijection to the images shown in Table~\ref{t:to_rlc}.
\end{proof}

\begin{table}[h]
\begin{tabular}{|p{0.3in}|p{3.0in}|}
\hline
Type & Set of root lattice coordinates \\
\hline
$\widetilde{C}/C$ & $(a_1, \ldots, a_n) \in \mathbb{Z}^n$ \\
\hline
$\widetilde{B}/B$ & $(a_1, \ldots, a_n) \in \mathbb{Z}^n$ such that $\sum_{i=1}^n |a_i|$ is even. \\
\hline
$\widetilde{B}/D$ & $(a_1, \ldots, a_n) \in \mathbb{Z}^n$ \\
\hline
$\widetilde{D}/D$ & $(a_1, \ldots, a_n) \in \mathbb{Z}^n$ such that $\sum_{i=1}^n |a_i|$ is even. \\
\hline
\end{tabular}
\caption{Root lattice points for $\widetilde{W}/W$}\label{t:to_rlc}
\end{table}

\begin{example}
For the minimal length coset representative $w=[-11,-9,-1,8,16,18] \in \tld{C}_3/C_3$, the root lattice coordinates $e_1+2e_2-2e_3$ can be read directly from the levels of the lowest beads in the first three runners of the abacus in Figure~\ref{fig:abacus}. 
\end{example}

Shi \cite{Shi-presentations} has worked out further details about the
relationship between root system geometry and mirrored
$\mathbb{Z}$-permutations.

\bigskip
\section{Core partitions}

\subsection{Definitions}

A {\bf partition} is a sequence $\lambda_1 \geq \lambda_2 \geq \cdots \geq
\lambda_k \geq 0$ of weakly decreasing integers.  Each partition has an
associated {\bf diagram} in which we place $\lambda_i$ unit boxes on the $i$-th
row of the diagram, where the first row is drawn at the top of the diagram.
The {\bf hook length} of a box $B$ in $\lambda$ is the sum of the number of boxes
lying to the right of $B$ in the same row and the number of boxes lying below
$B$ in the same column, including $B$ itself.  
The {\bf main diagonal} of a partition diagram is the set of all boxes 
with position coordinates $(i,i)\in \N^2$; for non-zero integers $j$, the 
{\bf $j$-th diagonal}
 of a partition diagram is the set of all boxes with position 
coordinates $(i,i+j)\in \N^2$.  We use the notation $p \p q$ to denote the 
integer in $\{0, 1, \ldots, q-1\}$ that is equal to $p$ mod $q$.

\begin{definition}
We say that a partition $\lambda$ is a {\bf$(2n)$-core} if it is impossible to
remove $2n$ consecutive boxes from the southeast boundary of the partition
diagram in such a way that the result is still a partition diagram.
Equivalently, $\lambda$ is a $(2n)$-core if no box in $\lambda$ has a hook
length that is divisible by $2n$.  We say that a partition $\lambda$ is {\bf
symmetric} if the length of the $i$-th row of $\lambda$ is equal to the length
of the $i$-th column of $\lambda$, for all $i$.  We say that a partition
$\lambda$ is {\bf even} if there are an even number of boxes on the main
diagonal of $\lambda$.
\end{definition}

Every abacus diagram $a$ determines a partition $\lambda$, as follows. 

\begin{definition}\label{def:Fc}
Given an abacus $a$, create a partition $\Fc(a)$ whose southeast boundary is
the lattice path obtained by reading the entries of the abacus in reading order
and recording a north-step for each bead, and recording an east-step for each
gap.

If we suppose that there are $M$ active beads in $a$, then $\Fc(a)$ can be
equivalently described as the partition whose $i$-th row contains the same number
of boxes as gaps that appear before the $(M-i+1)$st active bead in reading
order.  In this way each box in the partition corresponds to a unique bead-gap
pair from the abacus in which the gap occurs before the active bead in reading
order.
\end{definition}

For an active bead $b$ in an abacus $a$, we define the {\bf symmetric gap}
$g(b)$ to be the gap in position $2N-b$ that exists by
Lemma~\ref{l:ab_balance}.  Then, for $b > N$ the bead-gap pair $(b, g(b))$ in $a$
corresponds to the box on the main diagonal on the row of $\Fc(a)$
corresponding to $b$.

\begin{example}
For the minimal length coset representative $w \in \tld{C}_3/C_3$ whose base
window is \\ $[-11,-9,-1,8,16,18]$, 
 the partition $\lam=\Fc(w)=(10,9,6,5,5,3,2,2,2,1)$ is given in Figure~\ref{fig:coreex}.  To find this partition from the abacus diagram in Figure~\ref{fig:abacus}, follow the abacus in reading order, recording a horizontal step for every gap (starting with the gap in position $-4$) and a vertical step for every bead (ending with the bead in position $18$).
\begin{figure}[h]
\begin{center}
\epsfig{figure=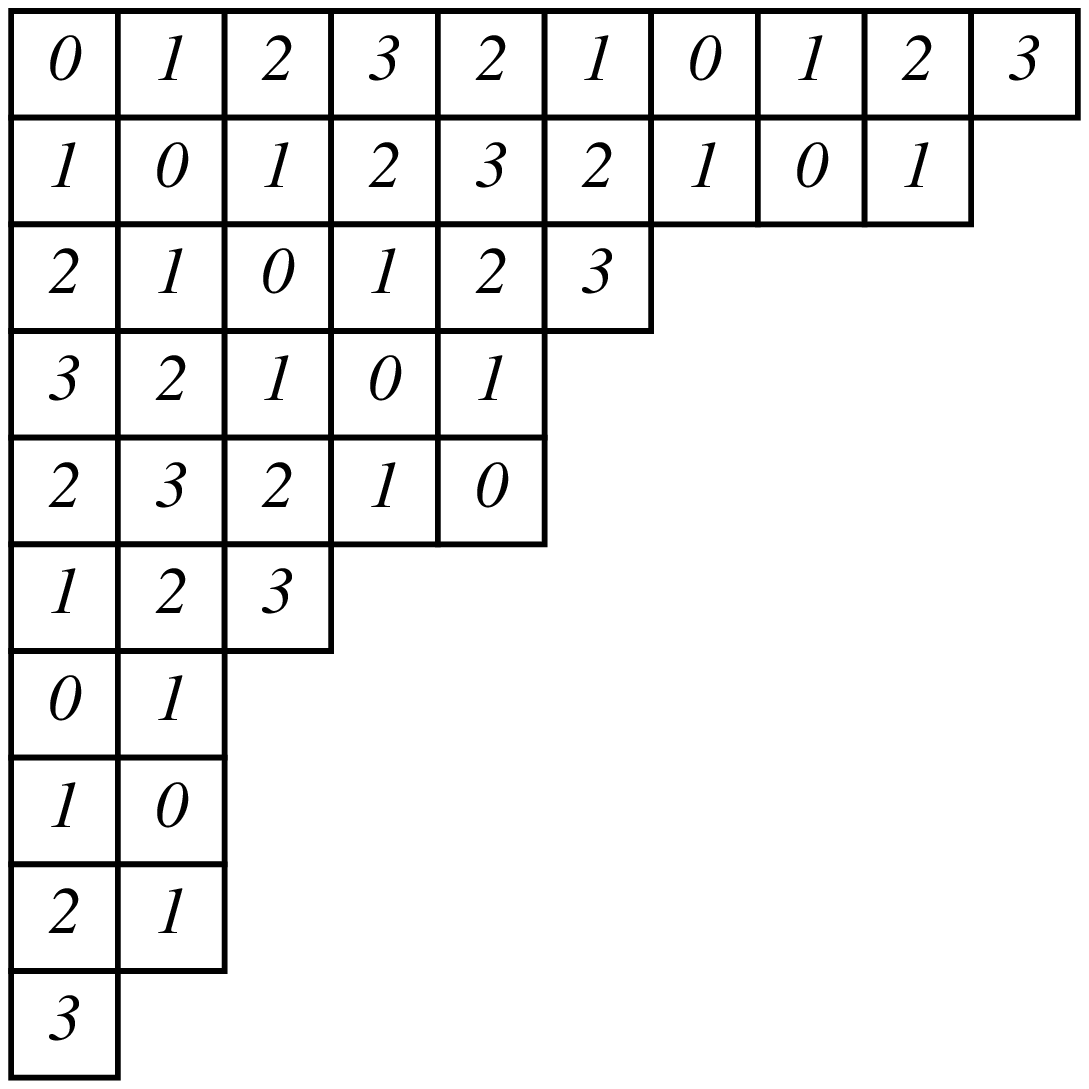, height=1.8 in}
\end{center}
\caption{The $6$-core partition $\lam=\Fc(w)=(10,9,6,5,5,3,2,2,2,1)$ for the minimal length coset representative $w=[-11,-9,-1,8,16,18]\in \tld{C}_3/C_3$.  The numbers inside the boxes are the residues, described below.}
\label{fig:coreex}
\end{figure}
\label{ex:core}
\end{example}

\begin{proposition}\label{p:bfabaci_to_cores}
The map $\Fc: \{\textup{abaci}\} \rightarrow \{\textup{partitions}\}$ 
is a bijection from and onto the sets shown in Table~\ref{t:main_results} on Page~\pageref{t:main_results}.
\end{proposition}
\begin{proof}
A balanced flush abacus $a$ determines a partition $\Fc(a)$ by
Definition~\ref{def:Fc}.  The fact that $a$ is flush by construction implies
that $\Fc(a)$ is a $(2n)$-core.  By Lemma~\ref{l:ab_balance}, the sequence of
gaps and beads is inverted when reflected about position $N$, so $\Fc(a)$ is
symmetric.

We can define an inverse map.  Starting from a $(2n)$-core partition $\lambda$,
encode its southeast boundary lattice path on the abacus by recording each
north-step as a bead and each east-step as a gap, placing the midpoint of the
lattice path from $\lambda$ to lie between entries $N-1$ and $N+1$ on the
abacus.  The resulting abacus will be flush because $\lambda$ is a $(2n)$-core.
The resulting abacus will be balanced because whenever position $i$ is the
lowest bead on runner $(i \p N)$, then by symmetry $2N-i$ is the highest gap on
runner $(N-i \p N)$ so $N-i$ is the lowest bead on runner $(N-i \p N)$.

Moreover, we claim that the map $\Fc$ restricts to a bijection between even
abaci and even core partitions.  In the correspondence between abaci and
partitions, the entry $N$ corresponds to the midpoint of the boundary lattice
path.  In particular, this entry lies at the corner of a box on the main
diagonal.  Therefore, the number of gaps preceding $N$ in the reading
order of $a$ is equal to the number of horizontal steps lying below the main
diagonal of $\lambda = \Fc(a)$, which is exactly the number of boxes
contained on the main diagonal of $\lambda$.
\end{proof}

\subsection{Residues for the action of $\widetilde{W}$}

If we translate the action of the Coxeter generators on abaci through the
bijection $\Fc$, we obtain an action of $\widetilde{W}$ on the symmetric
$(2n)$-core partitions.

To describe this action, we introduce the notion of a {\bf residue} for a box in
the diagram of a symmetric $(2n)$-core partition.  The idea that motivates the
following definitions is that $s_i$ should act on a symmetric $(2n)$-core
$\lambda$ by adding or removing all boxes with residue $i$.  In contrast with
the situation in types $\widetilde{A}$ and $\widetilde{C}$, it will turn out
that for types $\widetilde{B}$ and $\widetilde{D}$ the residue of a box in
$\lambda$ may depend on $\lambda$ and not merely on the coordinates of the box.

To begin, we orient $\N^2$ so that $(i,j)$ corresponds to row $i$ and column $j$
of a partition diagram, and define the {\bf fixed residue} of a position in
$\N^2$ to be
\[ \res(i,j) = \begin{cases}
    (j-i) \p (2n) & \text{ if } 0 \leq (j-i) \p (2n) \leq n \\
    2n - \big( (j-i) \p (2n) \big) & \text{ if } n < (j-i) \p (2n) < 2n. \\
\end{cases} \]
Then, the fixed residues are given by extending the pattern illustrated below.
\[ \myscalebox{0.5}{ \tightpartition{
\nn{0} & \nn{1} & \nn{2} & \nn{\cdots} & \nn{n-1} & \nn{n} & \nn{n-1} & \nn{\cdots} & \nn{2} & \nn{1} & \nn{0} & \nn{1} & \nn{2} \\
\nn{1} & \nn{0} & \nn{1} & \nn{2} & \nn{\cdots} & \nn{n-1} & \nn{n} & \nn{n-1} & \nn{\cdots} & \nn{2} & \nn{1} & \nn{0} & \nn{1} \\
\nn{2} & \nn{1} & \nn{0} & \nn{1} & \nn{2} & \nn{\cdots} & \nn{n-1} & \nn{n} & \nn{n-1} & \nn{\cdots} & \nn{2} & \nn{1} & \nn{0} \\
\nn{\vdots} & \nn{2} & \nn{1} & \nn{0} & \nn{1} & \nn{2} & \nn{\cdots} & \nn{n-1} & \nn{n} & \nn{n-1} & \nn{\cdots} & \nn{2} & \nn{1} \\
\nn{n-1} & \nn{\vdots} & \nn{2} & \nn{1} & \nn{0} & \nn{1} & \nn{2} & \nn{\cdots} & \nn{n-1} & \nn{n} & \nn{n-1} & \nn{\cdots} & \nn{2} \\
\nn{n} & \nn{n-1} & \nn{\vdots} & \nn{2} & \nn{1} & \nn{0} & \nn{1} & \nn{2} & \nn{\cdots} & \nn{n-1} & \nn{n} & \nn{n-1} & \nn{\cdots} \\
\nn{n-1} & \nn{n} & \nn{n-1} & \nn{\vdots} & \nn{2} & \nn{1} & \nn{0} & \nn{1} & \nn{2} & \nn{\cdots} & \nn{n-1} & \nn{n} & \nn{n-1} \\
\nn{\vdots} & \nn{n-1} & \nn{n} & \nn{n-1} & \nn{\vdots} & \nn{2} & \nn{1} & \nn{0} & \nn{1} & \nn{2} & \nn{\cdots} & \nn{n-1} & \nn{n} \\
} } \]

We define an {\bf escalator} to be a connected component of the entries $(i,j)$
in $\N^2$ satisfying 
\[ (j-i) \p (2n) \in \{n-1, n, n+1\}. \]
If an escalator lies above the main diagonal $i=j$, then we say it is an {\bf
upper escalator}; otherwise, it is called a {\bf lower escalator}.
Similarly, we define a {\bf descalator} to be a connected component of the
entries $(i,j)$ in $\N^2$ satisfying 
\[ (j-i) \p (2n) \in \{-1, 0, 1\}. \]
If a descalator lies above the main diagonal $i=j$, then we say it is an {\bf
upper descalator}; if a descalator lies below the main diagonal $i=j$, it is
called a {\bf lower descalator}.  There is one {\bf main descalator} that
includes the main diagonal $i=j$ which is neither upper nor lower.

Let $\lambda$ be a symmetric $(2n)$-core partition.  We define the {\bf residues
of the boxes in an upper escalator} lying on row $i$ depending on the number of
boxes in the $i$-th row of $\lambda$ that intersect the escalator, as shown in
Figure~\ref{f:calator}(a).

\begin{figure}[h]
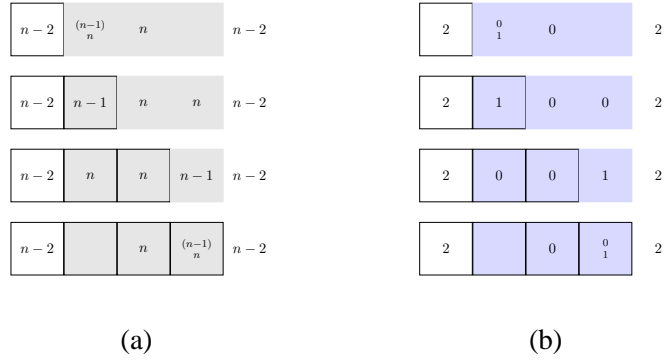

    \begin{tabular}{p{2in} p{2in}}
\myscalebox{0.5}{ \partition{ \np{n-2} & \ns{\ontop{(n-1)}{n}} & \ns{n} &
\ns{ } & \nn{n-2} \\ } } & \myscalebox{0.5}{ \partition{ \np{2} & \nb{\ontop{0}{1}} & \nb{0} & \nb{ } & \nn{2} \\ } } \\
\myscalebox{0.5}{ \partition{ \np{n-2} & \nsp{n-1} & \ns{n} & \ns{n} & \nn{n-2}
\\ } } & \myscalebox{0.5}{ \partition{ \np{2} & \nbp{1} & \nb{0} & \nb{0} & \nn{2} \\ } } \\
\myscalebox{0.5}{ \partition{ \np{n-2} & \nsp{n} & \nsp{n} & \ns{n-1} & \nn{n-2}
\\ } } & \myscalebox{0.5}{ \partition{ \np{2} & \nbp{0} & \nbp{0} & \nb{1} & \nn{2} \\ } } \\
\myscalebox{0.5}{ \partition{ \np{n-2} & \nsp{ } & \nsp{n} &
\nsp{\ontop{(n-1)}{n}} & \nn{n-2} \\ } } & \myscalebox{0.5}{ \partition{ \np{2} & \nbp{ } & \nbp{0} & \nbp{\ontop{0}{1}} & \nn{2} \\ } } \\
 & \\
 \hspace{0.6in} (a) & \hspace{0.6in} (b) \\
\end{tabular}
\caption{Residue assignments for upper escalators and descalators}\label{f:calator}
\end{figure}

Here, the outlined boxes represent entries that belong to the row of $\lambda$,
while the shaded cells represent entries in an upper escalator.  The schematic
in Figure~\ref{f:calator}(a) shows all ways in which these two types of entries can
overlap, and we have written the residue assignments that we wish to assign for
each entry.  

Note that in the first case, where $\lambda$ is adjacent to but does not
intersect the upper escalator, we view the first box of the escalator as being
simultaneously $(n-1)$-addable and $n$-addable.  In this case, the rightmost
cell of the upper escalator has undetermined residue, and is neither addable
nor removable.  A similar situation occurs when a row of $\lambda$ ends with
three boxes in the upper escalator.  In all other cases, the entries of an
upper escalator have undefined residue.

We similarly define the {\bf residues of the boxes in a lower escalator} lying
on column $j$ depending on the number of boxes in the $j$-th column of $\lambda$
that intersect the escalator.  The precise assignment is simply the transpose of
the schematic in Figure~\ref{f:calator}(a).

Moreover, we similarly define the {\bf residues of the boxes in an upper
descalator} on row $i$ depending on the number of boxes in the $i$-th row of
$\lambda$ that intersect the descalator, as shown in Figure~\ref{f:calator}(b).
The transpose of this schematic gives the assignment of {\bf residues of the
boxes in a lower descalator}.

Finally, the residues of the descalator containing the main diagonal are fixed,
and we define the {\bf residue of an entry $(i,j)$ lying on the main
descalator} to be
\[ \mres(i,j) = \begin{cases}
    0 & \text{ if } (j-i) = 0 \\
    1 & \text{ if } (j-i) \in \{1, -1\} \text{ and } (j+i) \equiv 1 \mod 4 \\
    0 & \text{ if } (j-i) \in \{1, -1\} \text{ and } (j+i) \equiv 3 \mod 4 \\
\end{cases} \]
where the upper left most box has coordinates $(i,j) = (1,1)$.

\begin{example}
In Figure~\ref{f:resex}, we consider residues using all of the features
discussed above.  It will turn out that this corresponds to type
$\widetilde{D}_5/D_5$ and that the assignment of residues in each of the other
types uses a subset of these features as described in the fourth column of
Table~\ref{t:main_results}.

Here, the unshaded entries of Figure~\ref{f:resex} are fixed residues given by
$\res(i,j)$ for $n = 5$; the gray shaded entries represent an upper and lower
escalator, and the blue shaded entries represent an upper and lower descalator.
The descalator containing the main diagonal has fixed residues given by
$\mres(i,j)$.  The precise residues of the boxes in the upper/lower
escalators/descalators for a particular partition depend on how the partition
intersects the shaded regions.

\begin{figure}[h!t]
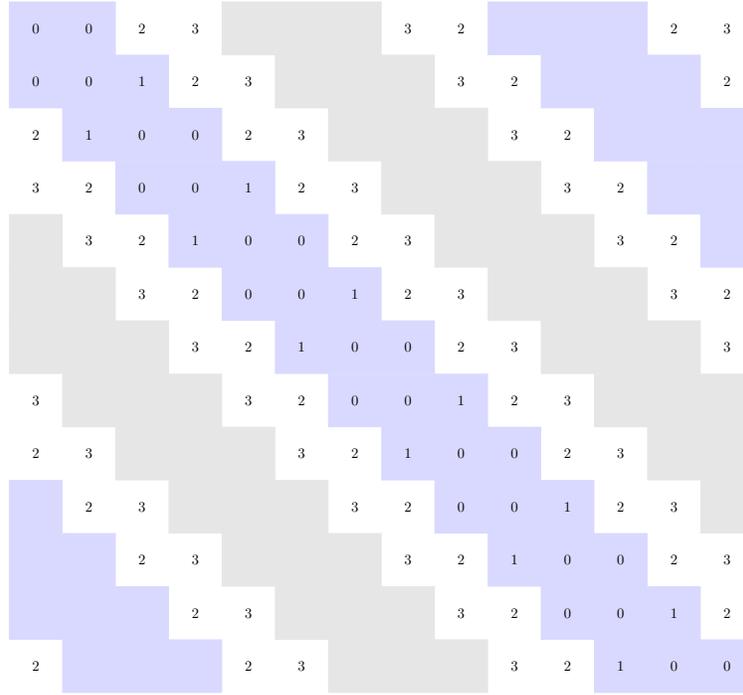

\myscalebox{0.5}{ \partition{
\nb{0} & \nb{0} & \nn{2} & \nn{3} & \ns{ } & \ns{ } & \ns{ } & \nn{3} & \nn{2} & \nb{ } & \nb{ } & \nb{ } & \nn{2} & \nn{3} \\
\nb{0} & \nb{0} & \nb{1} & \nn{2} & \nn{3} & \ns{ } & \ns{ } & \ns{ } & \nn{3} & \nn{2} & \nb{ } & \nb{ } & \nb{ } & \nn{2} \\
\nn{2} & \nb{1} & \nb{0} & \nb{0} & \nn{2} & \nn{3} & \ns{ } & \ns{ } & \ns{ } & \nn{3} & \nn{2} & \nb{ } & \nb{ } & \nb{ } \\
\nn{3} & \nn{2} & \nb{0} & \nb{0} & \nb{1} & \nn{2} & \nn{3} & \ns{ } & \ns{ } & \ns{ } & \nn{3} & \nn{2} & \nb{ } & \nb{ } \\ 
\ns{ } & \nn{3} & \nn{2} & \nb{1} & \nb{0} & \nb{0} & \nn{2} & \nn{3} & \ns{ } & \ns{ } & \ns{ } & \nn{3} & \nn{2} & \nb{ } \\
\ns{ } & \ns{ } & \nn{3} & \nn{2} & \nb{0} & \nb{0} & \nb{1} & \nn{2} & \nn{3} & \ns{ } & \ns{ } & \ns{ } & \nn{3} & \nn{2} \\
\ns{ } & \ns{ } & \ns{ } & \nn{3} & \nn{2} & \nb{1} & \nb{0} & \nb{0} & \nn{2} & \nn{3} & \ns{ } & \ns{ } & \ns{ } & \nn{3} \\
\nn{3} & \ns{ } & \ns{ } & \ns{ } & \nn{3} & \nn{2} & \nb{0} & \nb{0} & \nb{1} & \nn{2} & \nn{3} & \ns{ } & \ns{ } & \ns{ } \\
\nn{2} & \nn{3} & \ns{ } & \ns{ } & \ns{ } & \nn{3} & \nn{2} & \nb{1} & \nb{0} & \nb{0} & \nn{2} & \nn{3} & \ns{ } & \ns{ } \\
\nb{ } & \nn{2} & \nn{3} & \ns{ } & \ns{ } & \ns{ } & \nn{3} & \nn{2} & \nb{0} & \nb{0} & \nb{1} & \nn{2} & \nn{3} & \ns{ } \\
\nb{ } & \nb{ } & \nn{2} & \nn{3} & \ns{ } & \ns{ } & \ns{ } & \nn{3} & \nn{2} & \nb{1} & \nb{0} & \nb{0} & \nn{2} & \nn{3} \\
\nb{ } & \nb{ } & \nb{ } & \nn{2} & \nn{3} & \ns{ } & \ns{ } & \ns{ } & \nn{3} & \nn{2} & \nb{0} & \nb{0} & \nb{1} & \nn{2} \\
\nn{2} & \nb{ } & \nb{ } & \nb{ } & \nn{2} & \nn{3} & \ns{ } & \ns{ } & \ns{ } & \nn{3} & \nn{2} & \nb{1} & \nb{0} & \nb{0} \\
} } 
\caption{The fixed residues in $\widetilde{D}_5/D_5$}\label{f:resex}
\end{figure}
\end{example}

\begin{definition}\label{d:residue}
Suppose $\lambda$ is a symmetric $(2n)$-core partition representing an element
of type $\widetilde{W}_n/W_n \in \{\widetilde{C}_n/C_n, \widetilde{B}_n/B_n,
\widetilde{B}_n/D_n, \widetilde{D}_n/D_n\}$, and embed the digram of $\lambda$
in $\N^2$ as above.  Then we assign the {\bf residue} of $(i,j) \in \N^2$ as
described above, using the features listed in the fourth column of
Table~\ref{t:main_results}.
\end{definition}

\begin{table}[h!t]
\begin{tabular}{|p{0.3in}|p{1.7in}|p{1.7in}|p{1.9in}|}
\hline
Type & Abaci & Partitions & Features for residue assignment \\
\hline
$\widetilde{C}/C$ & balanced flush abaci & symmetric $(2n)$-cores & fixed residues only \\
\hline
$\widetilde{B}/B$ & even balanced flush abaci & even symmetric $(2n)$-cores & fixed residues with descalators \\
\hline
$\widetilde{B}/D$ & balanced flush abaci & symmetric $(2n)$-cores & fixed residues with escalators \\
\hline
$\widetilde{D}/D$ & even balanced flush abaci & even symmetric $(2n)$-cores & fixed residues with escalators and descalators \\
\hline
\end{tabular}
\caption{Core partitions for $\widetilde{W}/W$}\label{t:main_results}
\end{table}

\begin{definition}
Given a symmetric $(2n)$-core partition $\lambda$ with residues assigned, we say
that two boxes from $\N^2$ are {\bf $i$-connected} whenever they share an edge
and have the same residue $i$.  We refer to the $i$-connected components of
boxes from $\N^2$ as {\bf $i$-components}.

We say that an $i$-component $C$ is {\bf addable} if adding the boxes of $C$ to
the diagram of $\lambda$ results in a partition, and that $C$ is {\bf removable}
if removing the boxes of $C$ from the diagram of $\lambda$ results in a
partition.
\end{definition}

We are now in a position to state the action of $\widetilde{W}$ on symmetric
$(2n)$-core partitions in terms of residues.

\begin{theorem}\label{t:core_action}
Let $w \in \widetilde{W}/W$ and suppose $\lambda = \Fc(w)$.  If $s_i$ is an
ascent for $w$ then $s_i$ acts on $\lambda$ by adding all addable $i$-components
to $\lambda$.  If $s_i$ is a descent for $w$ then $s_i$ acts by removing all
removable $i$-components from $\lambda$.  If $s_i$ is neither an ascent nor a
descent for $w$ then $s_i$ does not change $\lambda$.
\end{theorem}

The proof of this result is postponed to Section~\ref{s:proofs}.

\begin{example}
Consider the minimal length coset representative $w=[-12, -7, -5, 2, 3, 8, 9, 16, 18, 23]\in \tld{D}_5/D_5$.  The corresponding core partition $\lam=\Fc(w)=(11,8,7,4,3,3,3,2,1,1,1)$ is pictured in Figure~\ref{fig:dd_ex}. The known residues are placed in their corresponding boxes.  

If we were to apply the generator $s_0$ to $\lam$, this would remove four
boxes---the two boxes in the upper right corner and the two boxes in the lower
left corner. The box on the main diagonal with residue $0$ is not removed
because we cannot remove only part of its connected component of residue $0$
boxes.

From $\lam$, we determine that $s_0$ and $s_4$ are descents (as they would
remove boxes), $s_1$ and $s_3$ are ascents (as they would add boxes), and $s_2$
and $s_5$ are neither ascents nor descents (as they would leave the diagram
unchanged).  
\end{example}

\begin{figure}[h]
\myscalebox{0.5}{\begin{tikzpicture}
\draw (0,0) node [rectangle, minimum size=1.5cm, inner sep=0pt, draw, anchor=south west, fill=blue!20] {$0$};
\draw (1.5cm,0) node [rectangle, minimum size=1.5cm, inner sep=0pt, draw, anchor=south west, fill=blue!20] {$0$};
\draw (3.0cm,0) node [rectangle, minimum size=1.5cm, inner sep=0pt, draw, anchor=south west] {$2$};
\draw (4.5cm,0) node [rectangle, minimum size=1.5cm, inner sep=0pt, draw, anchor=south west] {$3$};
\draw (6.0cm,0) node [rectangle, minimum size=1.5cm, inner sep=0pt, draw, anchor=south west, fill=gray!20] {};
\draw (7.5cm,0) node [rectangle, minimum size=1.5cm, inner sep=0pt, draw, anchor=south west, fill=gray!20] {};
\draw (9.0cm,0) node [rectangle, minimum size=1.5cm, inner sep=0pt, draw, anchor=south west, fill=gray!20] {};
\draw (10.5cm,0) node [rectangle, minimum size=1.5cm, inner sep=0pt, draw, anchor=south west] {$3$};
\draw (12.0cm,0) node [rectangle, minimum size=1.5cm, inner sep=0pt, draw, anchor=south west] {$2$};
\draw (13.5cm,0) node [rectangle, minimum size=1.5cm, inner sep=0pt, draw, anchor=south west, fill=blue!20] {$0$};
\draw (15.0cm,0) node [rectangle, minimum size=1.5cm, inner sep=0pt, draw, anchor=south west, fill=blue!20] {$0$};
\draw (16.5cm,0) node [rectangle, minimum size=1.5cm, inner sep=0pt, anchor=south west, fill=blue!20] {$1$};
\draw (0,-1.5cm) node [rectangle, minimum size=1.5cm, inner sep=0pt, draw, anchor=south west, fill=blue!20] {$0$};
\draw (1.5cm,-1.5cm) node [rectangle, minimum size=1.5cm, inner sep=0pt, draw, anchor=south west, fill=blue!20] {$0$};
\draw (3.0cm,-1.5cm) node [rectangle, minimum size=1.5cm, inner sep=0pt, draw, anchor=south west, fill=blue!20] {$1$};
\draw (4.5cm,-1.5cm) node [rectangle, minimum size=1.5cm, inner sep=0pt, draw, anchor=south west] {$2$};
\draw (6.0cm,-1.5cm) node [rectangle, minimum size=1.5cm, inner sep=0pt, draw, anchor=south west] {$3$};
\draw (7.5cm,-1.5cm) node [rectangle, minimum size=1.5cm, inner sep=0pt, draw, anchor=south west, fill=gray!20] {};
\draw (9.0cm,-1.5cm) node [rectangle, minimum size=1.5cm, inner sep=0pt, draw, anchor=south west, fill=gray!20] {$5$};
\draw (10.5cm,-1.5cm) node [rectangle, minimum size=1.5cm, inner sep=0pt, draw, anchor=south west, fill=gray!20] {$\ontop{4}{5}$};
\draw (12.0cm,-1.5cm) node [rectangle, minimum size=1.5cm, inner sep=0pt,  anchor=south west] {$3$};
\draw (13.5cm,-1.5cm) node [rectangle, minimum size=1.5cm, inner sep=0pt,  anchor=south west] {$2$};
\draw (15.0cm,-1.5cm) node [rectangle, minimum size=1.5cm, inner sep=0pt,  anchor=south west, fill=blue!20] {};
\draw (16.5cm,-1.5cm) node [rectangle, minimum size=1.5cm, inner sep=0pt,  anchor=south west, fill=blue!20] {};
\draw (0,-3.0cm) node [rectangle, minimum size=1.5cm, inner sep=0pt, draw, anchor=south west] {$2$};
\draw (1.5cm,-3.0cm) node [rectangle, minimum size=1.5cm, inner sep=0pt, draw, anchor=south west, fill=blue!20] {$1$};
\draw (3.0cm,-3.0cm) node [rectangle, minimum size=1.5cm, inner sep=0pt, draw, anchor=south west, fill=blue!20] {$0$};
\draw (4.5cm,-3.0cm) node [rectangle, minimum size=1.5cm, inner sep=0pt, draw, anchor=south west, fill=blue!20] {$0$};
\draw (6.0cm,-3.0cm) node [rectangle, minimum size=1.5cm, inner sep=0pt, draw, anchor=south west] {$2$};
\draw (7.5cm,-3.0cm) node [rectangle, minimum size=1.5cm, inner sep=0pt, draw, anchor=south west] {$3$};
\draw (9.0cm,-3.0cm) node [rectangle, minimum size=1.5cm, inner sep=0pt, draw, anchor=south west, fill=gray!20] {$4$};
\draw (10.5cm,-3.0cm) node [rectangle, minimum size=1.5cm, inner sep=0pt,  anchor=south west, fill=gray!20] {$5$};
\draw (12.0cm,-3.0cm) node [rectangle, minimum size=1.5cm, inner sep=0pt,  anchor=south west, fill=gray!20] {$5$};
\draw (0,-4.5cm) node [rectangle, minimum size=1.5cm, inner sep=0pt, draw, anchor=south west] {$3$};
\draw (1.5cm,-4.5cm) node [rectangle, minimum size=1.5cm, inner sep=0pt, draw, anchor=south west] {$2$};
\draw (3.0cm,-4.5cm) node [rectangle, minimum size=1.5cm, inner sep=0pt, draw, anchor=south west, fill=blue!20] {$0$};
\draw (4.5cm,-4.5cm) node [rectangle, minimum size=1.5cm, inner sep=0pt, draw, anchor=south west, fill=blue!20] {$0$};
\draw (6.0cm,-4.5cm) node [rectangle, minimum size=1.5cm, inner sep=0pt, anchor=south west, fill=blue!20] {$1$};
\draw (7.5cm,-4.5cm) node [rectangle, minimum size=1.5cm, inner sep=0pt, anchor=south west] {$2$};
\draw (9.0cm,-4.5cm) node [rectangle, minimum size=1.5cm, inner sep=0pt, anchor=south west] {$3$};
\draw (10.5cm,-4.5cm) node [rectangle, minimum size=1.5cm, inner sep=0pt, anchor=south west, fill=gray!20] {};
\draw (0,-6.0cm) node [rectangle, minimum size=1.5cm, inner sep=0pt, draw, anchor=south west, fill=gray!20] {};
\draw (1.5cm,-6.0cm) node [rectangle, minimum size=1.5cm, inner sep=0pt, draw, anchor=south west] {$3$};
\draw (3.0cm,-6.0cm) node [rectangle, minimum size=1.5cm, inner sep=0pt, draw, anchor=south west] {$2$};
\draw (4.5cm,-6.0cm) node [rectangle, minimum size=1.5cm, inner sep=0pt, anchor=south west, fill=blue!20] {$1$};
\draw (6.0cm,-6.0cm) node [rectangle, minimum size=1.5cm, inner sep=0pt, anchor=south west, fill=blue!20] {$0$};
\draw (0,-7.5cm) node [rectangle, minimum size=1.5cm, inner sep=0pt, draw, anchor=south west, fill=gray!20] {};
\draw (1.5cm,-7.5cm) node [rectangle, minimum size=1.5cm, inner sep=0pt, draw, anchor=south west, fill=gray!20] {};
\draw (3.0cm,-7.5cm) node [rectangle, minimum size=1.5cm, inner sep=0pt, draw, anchor=south west] {$3$};
\draw (4.5cm,-7.5cm) node [rectangle, minimum size=1.5cm, inner sep=0pt, anchor=south west] {$2$};
\draw (0,-9.0cm) node [rectangle, minimum size=1.5cm, inner sep=0pt, draw, anchor=south west, fill=gray!20] {};
\draw (1.5cm,-9.0cm) node [rectangle, minimum size=1.5cm, inner sep=0pt, draw, anchor=south west, fill=gray!20] {$5$};
\draw (3.0cm,-9.0cm) node [rectangle, minimum size=1.5cm, inner sep=0pt, draw, anchor=south west, fill=gray!20] {$4$};
\draw (4.5cm,-9.0cm) node [rectangle, minimum size=1.5cm, inner sep=0pt, anchor=south west] {$3$};
\draw (0,-10.5cm) node [rectangle, minimum size=1.5cm, inner sep=0pt, draw, anchor=south west] {$3$};
\draw (1.5cm,-10.5cm) node [rectangle, minimum size=1.5cm, inner sep=0pt, draw, anchor=south west, fill=gray!20] {$\ontop{4}{5}$};
\draw (3.0cm,-10.5cm) node [rectangle, minimum size=1.5cm, inner sep=0pt, anchor=south west, fill=gray!20] {$5$};
\draw (4.5cm,-10.5cm) node [rectangle, minimum size=1.5cm, inner sep=0pt, anchor=south west, fill=gray!20] {};
\draw (0,-12.0cm) node [rectangle, minimum size=1.5cm, inner sep=0pt, draw, anchor=south west] {$2$};
\draw (1.5cm,-12.0cm) node [rectangle, minimum size=1.5cm, inner sep=0pt,  anchor=south west] {$3$};
\draw (3.0cm,-12.0cm) node [rectangle, minimum size=1.5cm, inner sep=0pt, anchor=south west, fill=gray!20] {$5$};
\draw (0,-13.5cm) node [rectangle, minimum size=1.5cm, inner sep=0pt, draw, anchor=south west, fill=blue!20] {$0$};
\draw (1.5cm,-13.5cm) node [rectangle, minimum size=1.5cm, inner sep=0pt,  anchor=south west] {$2$};
\draw (0,-15.0cm) node [rectangle, minimum size=1.5cm, inner sep=0pt, draw, anchor=south west, fill=blue!20] {$0$};
\draw (1.5cm,-15.0cm) node [rectangle, minimum size=1.5cm, inner sep=0pt,  anchor=south west, fill=blue!20] {};
\draw (0,-16.5cm) node [rectangle, minimum size=1.5cm, inner sep=0pt, anchor=south west, fill=blue!20] {$1$};
\draw (1.5cm,-16.5cm) node [rectangle, minimum size=1.5cm, inner sep=0pt, anchor=south west, fill=blue!20] {};
\draw [line width=3pt] (0cm,-15.0cm) node (test1) [circle, fill=black, inner sep=0pt, minimum size=0pt, draw] { }
 -- ++(0:1.5cm)  node (test2) [circle, fill=black, inner sep=0pt, minimum size=0pt, draw] { }
 -- ++(90:1.5cm)  node (test3) [circle, fill=black, inner sep=0pt, minimum size=0pt, draw] { }
 -- ++(90:1.5cm)  node (test4) [circle, fill=black, inner sep=0pt, minimum size=0pt, draw] { }
 -- ++(90:1.5cm)  node (test5) [circle, fill=black, inner sep=0pt, minimum size=0pt, draw] { }
 -- ++(0:1.5cm)  node (test6) [circle, fill=black, inner sep=0pt, minimum size=0pt, draw] { }
 -- ++(90:1.5cm)  node (test7) [circle, fill=black, inner sep=0pt, minimum size=0pt, draw] { }
 -- ++(0:1.5cm)  node (test8) [circle, fill=black, inner sep=0pt, minimum size=0pt, draw] { }
 -- ++(90:1.5cm)  node (test9) [circle, fill=black, inner sep=0pt, minimum size=0pt, draw] { }
 -- ++(90:1.5cm)  node (tes4) [circle, fill=black, inner sep=0pt, minimum size=0pt, draw] { }
 -- ++(90:1.5cm)  node (tes5) [circle, fill=black, inner sep=0pt, minimum size=0pt, draw] { }
 -- ++(0:1.5cm)  node (tes6) [circle, fill=black, inner sep=0pt, minimum size=0pt, draw] { }
 -- ++(90:1.5cm)  node (tes7) [circle, fill=black, inner sep=0pt, minimum size=0pt, draw] { }
 -- ++(0:1.5cm)  node (tes8) [circle, fill=black, inner sep=0pt, minimum size=0pt, draw] { }
 -- ++(0:1.5cm)  node (te8) [circle, fill=black, inner sep=0pt, minimum size=0pt, draw] { }
 -- ++(0:1.5cm)  node (te7) [circle, fill=black, inner sep=0pt, minimum size=0pt, draw] { }
 -- ++(90:1.5cm)  node (te5) [circle, fill=black, inner sep=0pt, minimum size=0pt, draw] { }
 -- ++(0:1.5cm)  node (te6) [circle, fill=black, inner sep=0pt, minimum size=0pt, draw] { }
 -- ++(90:1.5cm)  node (t5) [circle, fill=black, inner sep=0pt, minimum size=0pt, draw] { }
 -- ++(0:1.5cm)  node (t6) [circle, fill=black, inner sep=0pt, minimum size=0pt, draw] { }
 -- ++(0:1.5cm)  node (t7) [circle, fill=black, inner sep=0pt, minimum size=0pt, draw] { }
 -- ++(0:1.5cm)  node (t8) [circle, fill=black, inner sep=0pt, minimum size=0pt, draw] { }
 -- ++(90:1.5cm)  node (est5) [circle, fill=black, inner sep=0pt, minimum size=0pt, draw] { };
\end{tikzpicture}}
\caption{The core partition $\lam=\Fc(w)=(11,8,7,4,3,3,3,2,1,1,1)$ corresponding to the minimal length coset representative $w=[-12, -7, -5, 2, 3, 8, 9, 16, 18, 23]\in \tld{D}_5/D_5$. The numbers indicate the residues of the boxes.}
\label{fig:dd_ex}
\end{figure}
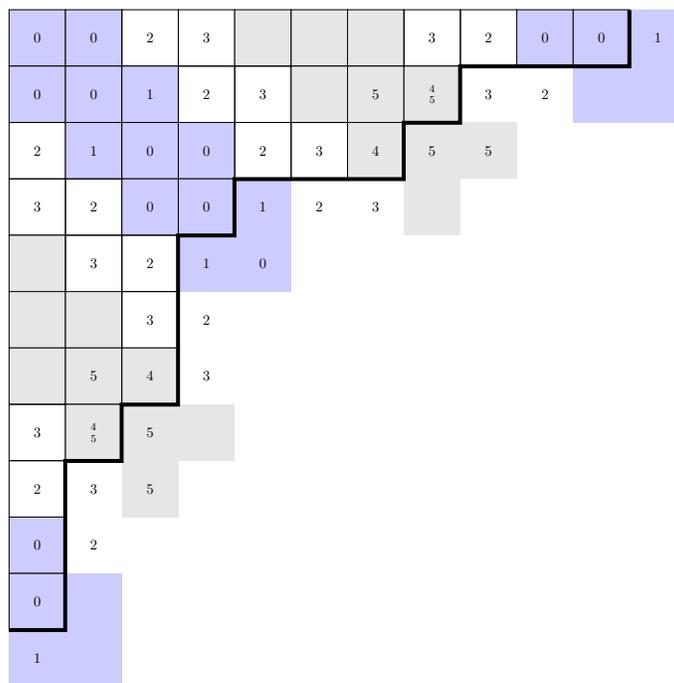

\subsection{Bruhat order on symmetric $(2n)$-cores}\label{s:bruhat}

In this section, we use an argument of Lascoux \cite{lascoux-cores} to show that
Bruhat order on the minimal length coset representatives $\widetilde{W}/W$
corresponds to a modified containment order on the corresponding core partition
diagrams.  This affirmatively answers a question of Billey and Mitchell \cite[Remark 12]{billey--mitchell}.

\begin{definition}\label{d:mcc}
Let $\lambda$ and $\mu$ be two symmetric $(2n)$-cores.  Suppose that every box
of $\mu$ that does not lie on an escalator or descalator is also a box of
$\lambda$, and that:
\begin{itemize}
    \item  Whenever the $i$-th row of $\mu$ intersects an upper escalator, upper
        descalator or the main descalator in $1$ box, then the $i$-th row of $\lambda$
        intersects the given region in $1$ or $3$ boxes.
    \item  Whenever the $i$-th row of $\mu$ intersects an upper escalator or upper
        descalator or the main descalator in $2$ boxes, then the $i$-th row of $\lambda$
        intersects the given region in $2$ or $3$ boxes.
    \item  Whenever the $i$-th row of $\mu$ intersects an upper escalator or upper
        descalator or the main descalator in $3$ boxes, then the $i$-th row of $\lambda$
        intersects the given region in $3$ boxes.
\end{itemize}
In this situation, we say that $\lambda$ {\bf contains} $\mu$, denoted $\lambda
\unrhd \mu$.
\end{definition}

\begin{theorem}\label{t:core_bruhat}
Let $w, x \in \widetilde{W}/W$.  Then $w \geq x$ in Bruhat order if and only if
$\Fc(w) \unrhd \Fc(x)$.
\end{theorem}
\begin{proof}
We proceed by induction on the number of boxes in $w$.  When the Coxeter length of $w$ is $1$,
then $x = w$, $x = e$, or $x$ is not related to $w$ in Bruhat order.  In each
case, the result is clear.

Let $s_i$ be a Coxeter generator such that $s_i w < w$ in $\widetilde{W}/W$.
If $x \leq w$ then the Lifting Lemma \cite[Proposition 2.2.7]{b-b} implies that
$s_i w \geq \min(x, s_i x)$ in $\widetilde{W}$.  Every reduced expression for
a nontrivial element of $\widetilde{W}/W$ ends in $s_0$, so if $s_i x < x$ then
$s_i x \in \widetilde{W}/W$.  Therefore, $s_i w \geq \min(x, s_i x)$ in
$\widetilde{W}/W$.

Conversely, if $w > s_i w \geq \min(x, s_i x)$ in $\widetilde{W}/W$ then we find
that $w \geq x$; this follows directly when $x < s_i x$ or $x = s_i x$, and
follows by another application of the Lifting Lemma when $s_i x < x$.

We therefore have the equivalence $x \leq w$ if and only if $\min(x, s_i x) \leq
s_i w$.  Hence, we reduce to considering the pair $w' = s_i w$, $x' = \min(x,
s_i x)$ in $\widetilde{W}/W$, and we need to show that $\Fc(w') \unrhd \Fc(x')$
if and only if $\Fc(w) \unrhd \Fc(x)$ to complete the proof by induction.


Suppose $\Fc(w) \unrhd \Fc(x)$.  Then we pass to $\Fc(s_i w)$ by removing
boxes with residue $i$ from the end of their rows and columns.  By
Definition~\ref{d:mcc}, every such box from $\Fc(w)$ is either removable in
$\Fc(x)$ or else absent from $\Fc(x)$.  Hence, $\Fc(w') \unrhd \Fc(x')$.
Similarly, it follows from Definition~\ref{d:mcc} that if $\Fc(w') \unrhd
\Fc(x')$ then every addable box with residue $i$ in $\Fc(x')$ is either addable
in $\Fc(w')$ or already present in $\Fc(w')$.  Hence, $\Fc(w) \unrhd \Fc(x)$.
\end{proof}

\bigskip
\section{Reduced expressions from cores}\label{s:upper_bounded_diagrams}

\subsection{The upper diagram}

Let $\lambda = \Fc(w)$.  We now define a recursive procedure to obtain a
canonical reduced expression for $w$ from $\lambda$.  Recall that the first
diagonal is the diagonal immediately to the right of the main diagonal.

\begin{definition}\label{def:cpp}
Define the {\bf reference diagonal} to be 
\[ \begin{cases} \text{the main diagonal } & \text{ in types $\widetilde{C}/C$ and $\widetilde{B}/D$ } \\
\text{the first diagonal } & \text{ in types $\widetilde{B}/B$ and $\widetilde{D}/D$ } \\
 \end{cases} \]

Given a core $\lam=\lam^{(1)}$ we define the {\bf central peeling procedure}
recursively as follows.  At step $i$, we consider $\lam^{(i)}$ and set
$d^{(i)}$ to be the number of boxes on the reference diagonal of
$\lam^{(i)}$.  Suppose $B_i$ is the box at the end of the $d^{(i)}$-th row of
$\lam^{(i)}$ and $r_i$ is the residue of this box.  In the case when $B_i$ is
both $n$-removable and $(n-1)$-removable, we set $r_i = n$.  Apply generator
$s_{r_i}$ to $\lam^{(i)}$ to find $\lam^{(i+1)}$.

We define $\Fr(\lam)$ to be the product of generators $s_{r_1}s_{r_2}\cdots
s_{r_\ell}$.

We also define the {\bf upper diagram}, denoted $U_\lam$, to be the union of the
boxes $B_i$ encountered in the central peeling procedure respecting the
following conditions: In $\tld{B}/D$ and $\tld{D}/D$, the application of
$s_n^D$ removes two boxes from the $d^{(i)}$-th row and we record in $U_\lam$ the
box {\em not} on the $n$-th diagonal; In $\tld{B}/B$ and $\tld{D}/D$, the
application of $s_0^D$ removes two boxes from the $d^{(i)}$-th row and we record
in $U_\lam$ the box {\em not} on the main or $2n$-th diagonal.
\end{definition}

\begin{proposition}\label{p:peeling}
We have that $\Fr(\lam)$ is a reduced expression for $w$ and the number of boxes
in $U_\lam$ equals the Coxeter length of $\Fw(\lam)$. 
\end{proposition}
\begin{proof}
At each step, the central peeling procedure records a descent so this follows
from Proposition~\ref{t:core_action}.
\end{proof}

We now present two examples of the central peeling procedure, one in $\tld{C}_3/C_3$ and another in $\tld{D}_4/D_4$.

\begin{example} \label{ex:peeling}
The steps of the central peeling procedure applied to $\lam=\Fc([-11,-9,-1,8,16,18]) \in \tld{C}_3/C_3$ are presented in Figure~\ref{fig:peeling}. The collection of gray boxes tallied during the central peeling procedure is $U_\lam$; Figure~\ref{fig:super} shows $U_\lam$ superimposed over $\lam$.  We can read off the canonical reduced expression $\Fr(\lam)$ starting in the center of $\lam$ and working our way up, reading the gray boxes from right to left.  We have $\Fr(\lam)=s_0s_1s_0s_3s_2s_1s_0s_2s_3s_2s_1s_0s_2s_3s_2s_1s_0$.
\end{example}

\begin{example} Let $w=[-15, -11, -10, 4, 5, 19, 20, 24]\in\tld{D}_4/D_4$. The steps of the central peeling procedure applied to $\lam=\Fc(w)$ are presented in Figure~\ref{fig:peelingD}; notice that no tallied boxes are on the main diagonal or the fourth diagonal. These gray boxes make up $U_\lam$, which in turn is superimposed over $\lam$ in Figure~\ref{fig:superD}.  The canonical reduced expression is $\Fr(\lam)=s_0s_4s_2s_1s_4s_3s_2s_0s_4s_3s_2s_1s_4s_3s_2s_0$.
\end{example}

\clearpage

\begin{figure}[t]
\epsfig{figure=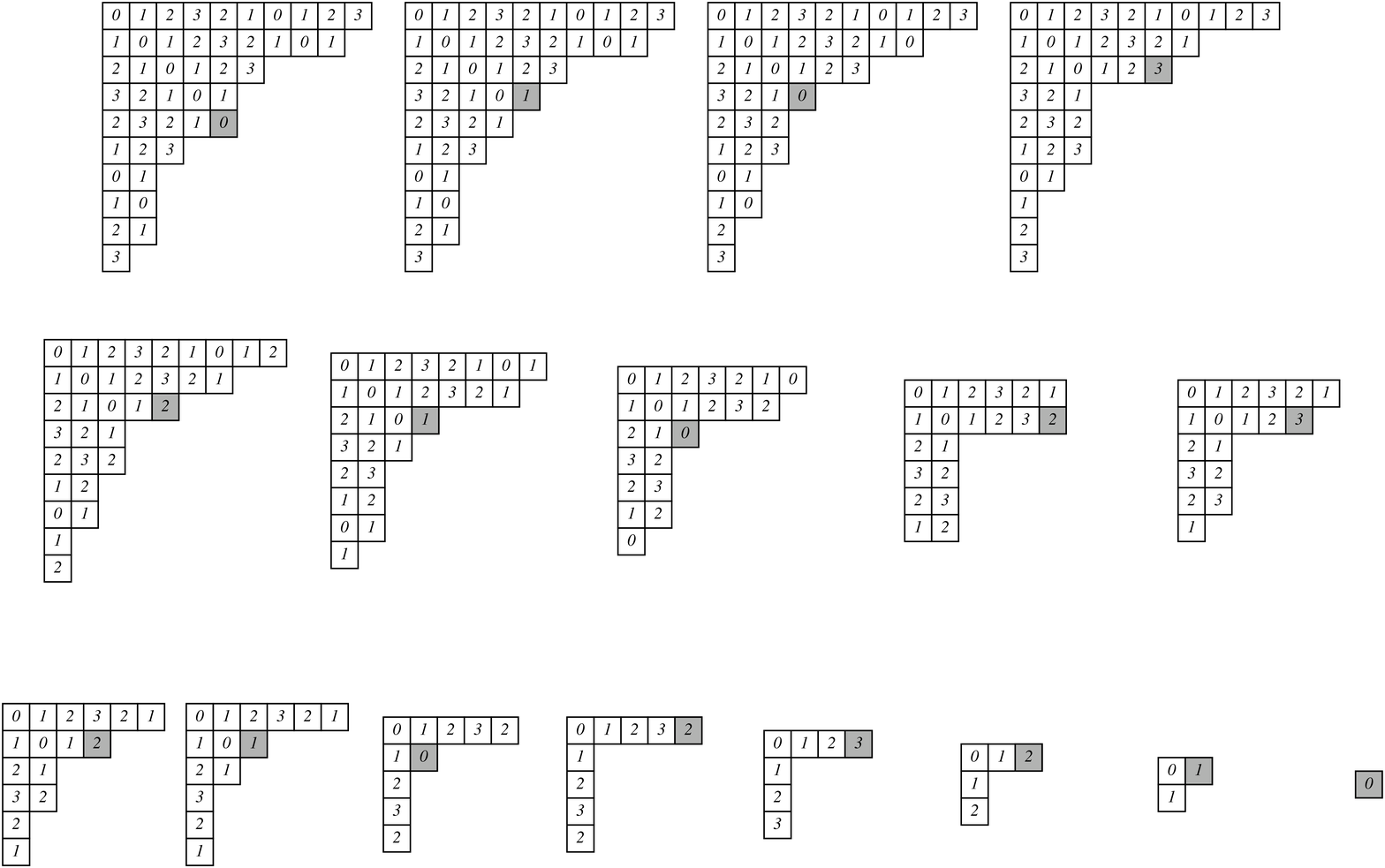, width=6in}
\caption{The application of the central peeling procedure described in Definition~\ref{def:cpp} on the core partition $\lam=\Fc(w)$ for $w=[-11,-9,-1,8,16,18]\in \tld{C}_3/C_3$.  The shaded boxes are the boxes $B_i$ tallied by the algorithm; they make up the upper diagram $U_\lam$, shown in Figure~\ref{fig:super}.}
\label{fig:peeling}
\end{figure}

\begin{figure}[t]
\begin{center}
\epsfig{figure=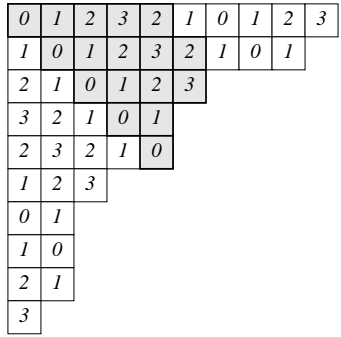, width=1.8in}
\end{center}
\caption{The core partition $\lam=\Fc(w)$ for $w=[-11,-9,-1,8,16,18]\in \tld{C}_3/C_3$, with its upper diagram $U_\lam$ superimposed.  Reading the shaded boxes from the bottom to the top from right to left gives the canonical reduced expression $\Fr(\lam)=s_0s_1s_0s_3s_2s_1s_0s_2s_3s_2s_1s_0s_2s_3s_2s_1s_0$.}
\label{fig:super}
\end{figure}
    
\clearpage

\begin{figure}[t]
\myscalebox{0.18}{
}
\caption{The application of the central peeling procedure described in Definition~\ref{def:cpp} on the core partition $\lam=\Fc(w)$ for $w=[-15, -11, -10, 4, 5, 19, 20, 24]\in \tld{D}_4/D_4$.  The shaded boxes are the boxes $B_i$ tallied by the algorithm; they make up the upper diagram $U_\lam$, shown in Figure~\ref{fig:superD}.}
\label{fig:peelingD}
\end{figure}

\begin{figure}[t]
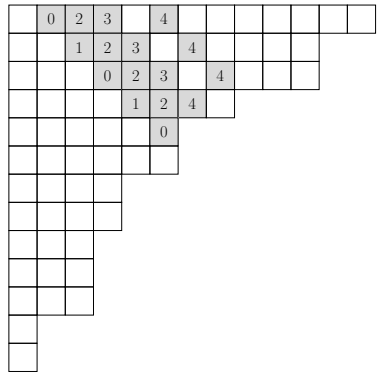

\begin{center}
\myscalebox{0.25}{%
}
\end{center}
\caption{The core partition $\lam=\Fc(w)$ for $w=[-15, -11, -10, 4, 5, 19, 20, 24]\in \tld{D}_4/D_4$, with its upper diagram $U_\lam$ shaded.  ($U_\lam$ was calculated in Figure~\ref{fig:peelingD}.) Reading the shaded boxes from the bottom to the top from right to left gives the canonical reduced expression  $\Fr(\lam)=s_0s_4s_2s_1s_4s_3s_2s_0s_4s_3s_2s_1s_4s_3s_2s_0$.}
\label{fig:superD}
\end{figure}

\clearpage

\subsection{The bounded diagram}\label{s:bounded_diagram}

We now describe a nonrecursive method to determine the upper diagram of any
core $\lam$.  This method generalizes a bijection of Lapointe and Morse
\cite{LM} in $\tld{A}_n/A_n$.  

We say that a box in $\lambda$ having hook length $< 2n$ is {\bf skew}.  The
skew boxes on any particular row form a contiguous segment lying at the end of
the row.  Consider the collection of boxes $\nu \subset \lam$ defined row by
row as the segment that begins at the box lying on the main diagonal and then
extends to the right for the same number of boxes as the number of skew boxes
in the row.  In $\tld{B}/B$ and $\tld{D}/D$, remove from $\nu$ all boxes along
the main diagonal.
In $\tld{B}/D$ and $\tld{D}/D$, remove from $\nu$ all boxes along the $n$-th
diagonal.  We call this collection of boxes $\nu$ the {\bf bounded diagram} of
$\lambda$, denoted $\widetilde{U}_\lam$.

\begin{theorem}\label{t:upperfromlam}
Fix a symmetric $(2n)$-core partition $\lam$.  Then, the bounded diagram
$\widetilde{U}_\lam$ is equal to the upper diagram $U_{\lam}$.
\end{theorem}

The proof of this theorem is postponed to Section~\ref{s:upper_partition}.

\begin{example}
For the minimal length coset representative $w=[-11,-9,-1,8,16,18] \in \tld{C}_3/C_3$ and its corresponding core partition $\lam=\Fc(w)$, a visualization of the construction of $\tld{U}_\lam$ is given in Figure~\ref{fig:upper}.  This agrees with $U_\lam$ as found in Figure~\ref{fig:super}.

\begin{figure}[h]
\begin{center}
\epsfig{figure=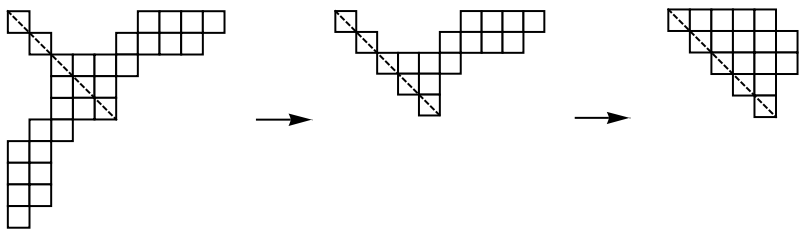, width=6in}
\end{center}
\caption{For $w=[-11,-9,-1,8,16,18]\in\tld{C}_3/C_3$, we left-justify the skew
boxes to the boxes on the main diagonal to find the bounded diagram. }
\label{fig:upper}
\end{figure}
\end{example}

\bigskip
\section{Bounded partitions}\label{s:bounded}

If we left-adjust the boxes of the upper diagram along with the residues they
contain, we obtain a structure that has appeared in other contexts including:
Young walls in crystal bases of quantum groups \cite{HK}, Eriksson and
Eriksson's partitions in \cite{Erik2}, the affine partitions of Billey and
Mitchell \cite{billey--mitchell}, and the $k$-bounded partitions of Lapointe and
Morse \cite{LM}.  Lam, Schilling and Shimozono \cite{lam--schilling--shimozono}
develop combinatorics analogous to the $k$-bounded partitions for type
$\widetilde{C}$, and \cite{spon} contains similar constructions for types
$\widetilde{B}$ and $\widetilde{D}$.  In this section, we explain how these
objects are related to the abacus diagrams and core partitions we have
introduced.

\begin{definition}\label{def:bpp}
Given a core $\lam$, apply the central peeling procedure to find the upper
diagram $U_\lam$.  We define the {\bf bounded partition} $\Fb(\lam)$ to be the
partition $\beta$ whose $i$-th part equals the number of boxes in row $i$ of $U_\lam$.  

In $\tld{B}/D$ (and $\tld{D}/D$), if there exists a part of size $n$
($n-1$, respectively), and the last part $\beta_k$ of this size has rightmost box
with residue $n-1$, then adorn $\beta_k$ with a star decoration.
\end{definition}

By construction, the number of boxes in a bounded partition is the Coxeter
length of the corresponding element in $\tld{W}/W$.  Besides being historical,
the motivation for the name ``bounded partition'' is that part sizes in $\Fb(w)$
are bounded in the different types by $2n$, $2n-1$, or $2n-2$, as shown in
Table~\ref{t:bdd}.

\begin{remark}
The need for a decoration in $\tld{B}/D$ and $\tld{D}/D$ arises because it is
possible that two core partitions yield the same bounded partition.  For
example, consider elements $w_1=s_0s_1\cdots s_{n-2}s_{n-1}$ and
$w_2=s_0s_1\cdots s_{n-2}s_n$ of $\tld{B}/D$.  They both have the same Coxeter
length and correspond to a bounded partition with one part of size $n$; however,
these elements are distinct and have different abaci and cores.  We would say
that $\Fb(w_1)=(n^*)$ and $\Fb(w_2)=(n)$.  The reader should interpret the star
as arising from a length $n$ sequence of generators (length $n-1$ in type
$\tld{D}/D$) that ends with $s_{n-1}$ instead of $s_n$.
\end{remark}

Recall that Definition~\ref{def:Fc} gives a correspondence between rows of
$\lam$ and active beads in $\Fa(\lam)$.  Together with
Theorem~\ref{t:upperfromlam}, this presents a method to determine the bounded
partition from an abacus diagram.  

\begin{lemma}\label{l:max}
For a row $r$ of $\lam$ corresponding to an active bead $b>N$ in $\Fa(\lam)$,
the number of skew boxes in row $r$ on or above the main diagonal equals the
number of gaps between $b$ and $\max(b-N,g(b))$, inclusive.
\end{lemma}
\begin{proof}
By definition~\ref{def:Fc}, the skew boxes on a row corresponding to an active
bead $b$ are the gaps between $b-N$ and $b$.  Also, the diagonal box on this
row corresponds to the symmetric gap $g(b)$, and the result follows.
\end{proof}

Define {\bf offsets} $x_0$ and $x_n$ corresponding to the occurrence of
generators $s_0^D$ and $s_n^D$ as follows.  These offsets record whether a fork
occurs on the left or right sides of the Coxeter graph, respectively.
$$x_0= 
\begin{cases}
0 & \textup{in types $\tld{C}/C$ and $\tld{B}/D$} \\
-1 & \textup{in types $\tld{B}/B$ and $\tld{D}/D$}. \\
\end{cases} \qquad 
x_n= 
\begin{cases}
0 & \textup{in types $\tld{C}/C$ and $\tld{B}/B$} \\
-1 & \textup{in types $\tld{B}/D$ and $\tld{D}/D$}. \\
\end{cases} 
$$

\begin{proposition}
\label{p:boundedfromabacus}
Given an abacus diagram $a$, the bounded partition $\Fb(a)$ is found by creating one part for each bead $b>N$ in descending order, as follows:
\begin{itemize}
\item For each bead $b>N+n$, create a part of size equal to the number of gaps
    between $b$ and $b-N$ plus $1+x_0+x_n$.
\item For each bead $N+1\leq b\leq N+n$, create a part of size equal to $b-N$
    plus $x_0$.    If $b=N+n$ in types $\tld{B}/D$ and $\tld{D}/D$, then star the resulting part.
\end{itemize}
\end{proposition}

\begin{proof}
When $b>N+n$, the number of skew boxes equals the number of gaps between $b$ and
$b-N$ by Lemma~\ref{l:max}.  We add one for the main diagonal, and apply offsets
for the main diagonal and $n$-th diagonal.  When $N+1\leq b\leq N+n$, the number
of boxes in row $r$ equals the number of gaps between $b$ and $g(b)$ by
Lemma~\ref{l:max}, which equals $b-N$ by Lemma~\ref{l:ab_balance}.  None of
these parts enter the $n$-th diagonal, so we only apply offset $x_0$.  Starring
conditions follow from Definition~\ref{def:bpp}.
\end{proof}

\begin{table}[t]
\begin{tabular}{|p{0.3in}|p{5.5in}|}
\hline
Type & Bounded partition structure \\ 
\hline
$\widetilde{C}/C$ & parts with size $\leq 2n$, where parts of size $1, \ldots,
n$ may occur at most once. \\ \hline
$\widetilde{B}/B$ & parts with size $\leq 2n-1$, where parts of size $1, \ldots,
n-1$ may occur at most once. \\ \hline
$\widetilde{B}/D$ & parts with size $\leq 2n-1$, where parts of size $1, \ldots,
n-1$ may occur at most once, and one of the parts of size $n$ may be starred.
\\ \hline
$\widetilde{D}/D$ & parts with size $\leq 2n-2$, where parts of size $1, \ldots,
n-2$ may occur at most once, and one of the parts of size $n-1$ may be starred.
\\ \hline
\end{tabular}
\caption{Bounded partitions of $\widetilde{W}/W$}\label{t:bdd}
\end{table}

We are now in a position to prove the main result of this section.

\begin{theorem}
\label{t:isbounded}
The map $\Fb$ is a bijection of $\widetilde{W}/W$ onto the partitions described in Table~\ref{t:bdd}.
\end{theorem}

\begin{proof}
We present the inverse map to $\Fb$ as described in
Proposition~\ref{p:boundedfromabacus}.  This inverse takes a bounded partition
$\beta$ and gives an abacus $a=\Fa(\beta)$.  We write
$\beta=(\beta_1,\hdots,\beta_{K},\beta_{K+1},\hdots,\beta_{K+k})$, where each
part $\beta_i\geq n+1+x_0+x_n$ for $1\leq i\leq K$ (and unstarred if
applicable), and $\beta_i\leq n+x_0$ for $K+1\leq i\leq K+k$ (where $n+x_0$ is
starred in types $\tld{B}/D$ and $\tld{D}/D$).

For each part $\beta_{K+i}$ for $1\leq i\leq k$, place a bead in position $N+\beta_{K+i}-x_0$. In types $\tld{B}/B$ and $\tld{D}/D$, if $K+k$ is odd, place a bead in position $N+1$.  After this, if position $N+j$ for $1\leq j\leq n$ is a gap, place a bead in position $N-j$. 

We now insert beads into the abacus one at time for parts $\beta_K$ through $\beta_1$ in this order. For each part, we consider the possible positions for placing the next bead to be the positions in reading order after the current lowest bead that have a bead one level above.  For part $\beta_K$ of size $n+j+x_0+x_n$, we place a bead in the $j$-th possible position.  For a part $\beta_i$ for $i<K$, place a bead $b$ in the $\beta_i-\beta_{i+1}+1$-th possible position.   By this construction, the number of gaps between $b$ and $b-N$ is exactly $\beta_i+1$.

Once we have finished placing beads as described, fill in beads above $n+1$ as necessary to make the abacus balanced and flush.

Observe that the abacus enforces the structural conditions given in Table~\ref{t:bdd}.
\end{proof}

It is natural to fill the boxes of the bounded partitions with the residues
present in the corresponding boxes of the upper diagrams before left adjusting.
As these residues are inherited from the residues in the core partition, they
follow a predictable pattern.

\begin{proposition}
Fix a bounded partition $\beta = \Fb(w)$ and fill the boxes of $\beta$ with
residues as described in Table~\ref{t:residues}.  The canonical reduced
expression $\Fr(w)$ is obtained by reading these residues from $\beta$, right
to left in rows and from bottom to top. 
\label{p:pattern}
\end{proposition}
\begin{proof}
This follows from Definition~\ref{def:bpp} and Theorem~\ref{t:upperfromlam}.
\end{proof}

\begin{table}[t]
\begin{tabular}{|p{0.3in}|p{5.5in}|}
\hline
Type  & Residue fillings \\ 
\hline
$\widetilde{C}/C$ &Fill all boxes in column $1\leq i\leq n+1$ with residue $i-1$ and all boxes in column $n+2\leq i\leq 2n$ with residue $2n+1-i$.
\\ \hline
$\widetilde{B}/B$ &Fill all boxes in columns $1$ and $2n-1$ with residues $0$ and $1$ alternating, starting at the top with $0$.  Fill all boxes in column $2\leq i\leq n$ with residue $i$ and all boxes in column $n+1 \leq i \leq 2n-2$ with residue $2n-i$.
\\ \hline
$\widetilde{B}/D$ & Fill all boxes in column $1\leq i\leq n-1$ with residue $i-1$, all boxes in column $n+1$ with residue $n$, and all boxes in column $n+2\leq i\leq 2n-1$ with residue $2n-i$.  In the last row of length $n$ (if one exists), if the part is starred, fill the box in column $n$ with residue $n-1$; otherwise fill it with residue $n$.  In other rows of length $n$, place residues $n$ and $n-1$ alternating up from there.  Fill all other boxes in column $n$ with residue $n-1$.\\ \hline
$\widetilde{D}/D$ & Fill all boxes in columns $1$ and $2n-2$ with residues $0$ and $1$ alternating, starting with $0$ at the top.  Fill all boxes in column $2\leq i\leq n-2$ with residue $i$, all boxes in column $n$ with residue $n$, and all boxes in column  $n+1 \leq i \leq 2n-3$ with residue $2n-i-1$.  In the last row of length $n-1$ (if one exists), if the part is starred, fill the box in column $n-1$ with residue $n-1$; otherwise fill it with residue $n$.  In other rows of length $n-1$, place residues $n$ and $n-1$ alternating up from there.  Fill all other boxes in column $n-1$ with residue $n-1$.
\\ \hline
\end{tabular}
\caption{Fillings of bounded partitions}\label{t:residues}
\end{table}

\begin{example}
The bounded partition $\beta=(8,8,5,5,5,4,2)$ could be a member of any one of
$\tld{B}_5/B_5$, $\tld{B}_5/D_5$, or $\tld{D}_5/D_5$.  The
different fillings given by Proposition~\ref{p:pattern} are shown in
Figure~\ref{fig:fillings}.  Since no part of $\beta$ is starred, the last row of
length $5$ in $\tld{B}_5/D_5$ contains a residue of $5$ instead of a $4$, and
the residues in this column are determined from there.  Similarly, the row of
length $4$ in $\tld{D}_5/D_5$ contains a residue of $5$.  To read the reduced
expression for $\Fr(\beta)$, read from right to left in rows from bottom to top;
for example in $\tld{D}_5/D_5$ we would have
$$\Fr(\beta)=s_2s_0s_5s_3s_2s_1s_5s_4s_3s_2s_0s_5s_4s_3s_2s_1s_5s_4s_3s_2s_0s_1s_2s_3s_5s_4s_3s_2s_1s_0s_2s_3s_5s_4s_3s_2s_0.$$
\begin{figure}[h]
\begin{tabular}{ccc}
\epsfig{figure=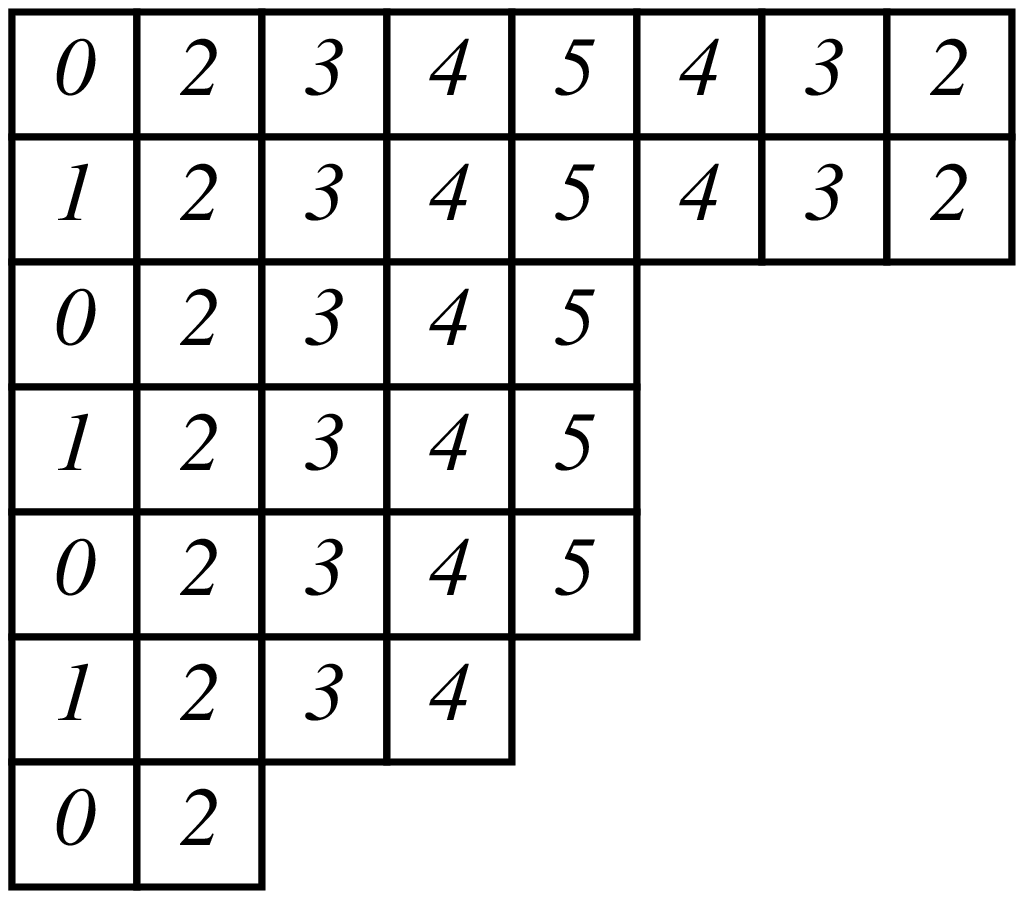,width=1.4in} &
\epsfig{figure=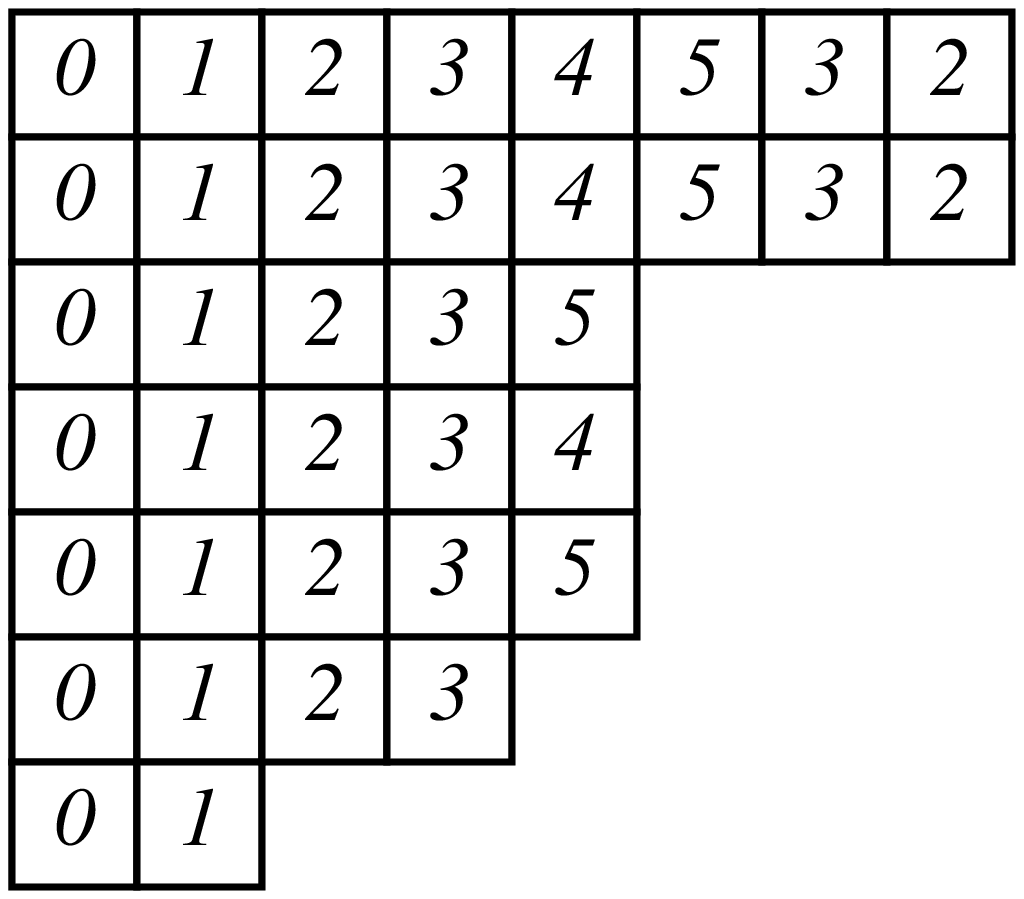,width=1.4in} &
\epsfig{figure=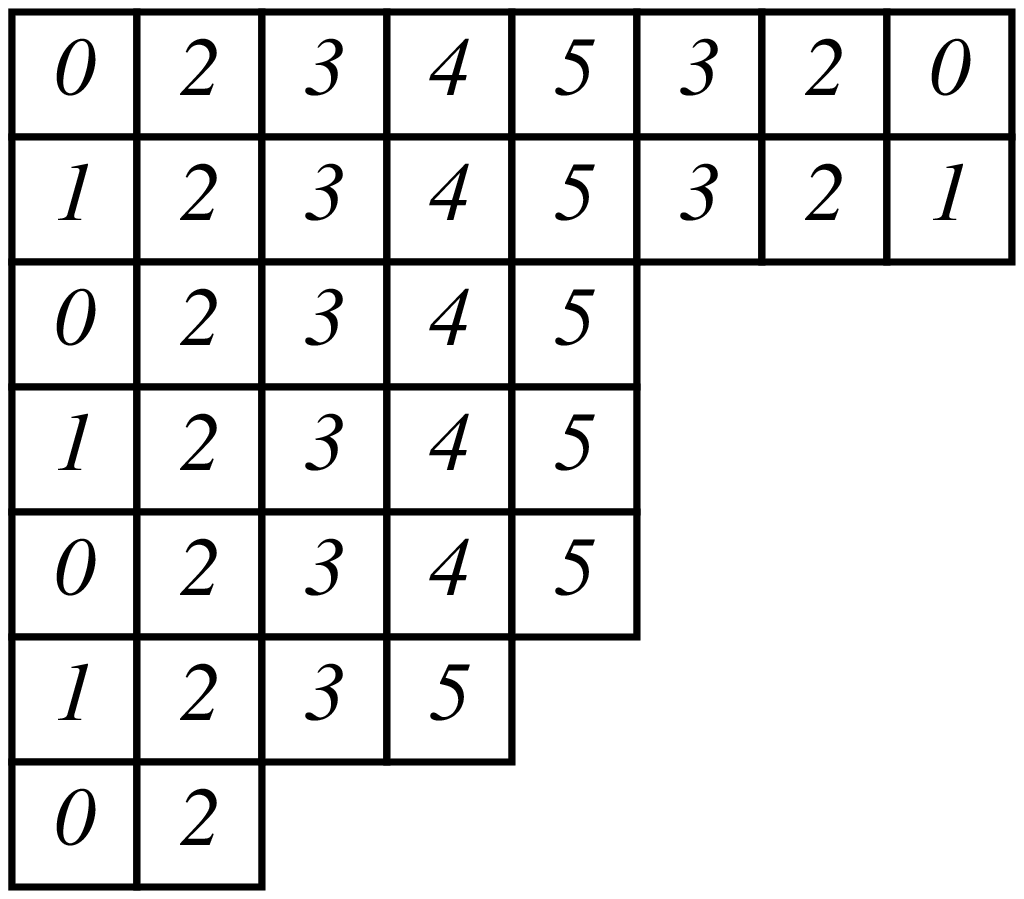,width=1.4in}
\end{tabular}
\caption{Example of the bounded partition $\beta=(8,8,5,5,5,4,2)$ and the residues
which populate $\beta$ when $\beta$ represents an element of $\tld{B}_5/B_5$, $\tld{B}_5/D_5$, and $\tld{D}_5/D_5$,
respectively.}\label{fig:fillings}
\end{figure}
\end{example}

\section{Coxeter length formulas}\label{s:coxeter_length_formulas}

If we sum contributions to the Coxeter length from beads in the same runners in
Proposition~\ref{p:boundedfromabacus}, we obtain the following formula for the
Coxeter length of a minimal length coset representative defined by way of an
abacus.  

\begin{corollary}\label{p:e_cox}
Let $a$ be the abacus corresponding to a minimal length coset representative $w\in\tld{W}/W$.  For $1\leq i\leq n$, choose the lowest bead $B(i)$ occurring in either runner $i$ or $N-i$, and the bead $n+1\leq b(i)\leq N+n$ occurring in $B(i)$'s runner.  Let $g_i$ be the number of gaps between $B(i)$ and $b(i)$ in $a$.  Then 
$$\ell(w)=\sum_{1\leq i\leq n} g_i+(1+x_0+x_n)\big(\textup{$\#$ of beads $>N+n$}\big) + \sum_{N+1\leq b\leq N+n} (b-N+x_0).$$
\end{corollary}

\begin{example}
In our example of $w=[-11,-9,-1,8,16,18]\in\tld{C}_3/C_3$ with abacus
$a=\Fa(w)$ in Figure~\ref{fig:abacus}, we create a part of the bounded
partition $\Fb(w)$ for each bead in $a$ greater than $N=7$.  The beads in
question are $18$, $16$, $11$, $9$, and $8$.  The first three beads are greater
than $7+3$, so we count the number of gaps between $b$ and $b-N$ and add one,
giving $5$, $5$, and $4$.  The last two beads are between $8$ and $10$, so we
take $b-N$, giving $2$ and $1$.  This agrees with the previously found
$\Fb(w)=(5,5,4,2,1)$.  In the terminology of Corollary~\ref{p:e_cox}, $B_1=b_1=8$, $B_2=16$, $b_2=9$, $B_3=18$, and $b_3=4$.  We then calculate $g_1=0$, $g_2=4$, $g_3=7$ and see that $\ell(w)=11+(1)(3)+3=17$, agreeing with the number of generators in $\Fr(w)$ found in Example~\ref{ex:peeling}.
\end{example}

Given a core partition $\lam$, Theorem~\ref{t:upperfromlam} gives a simple method to find the corresponding bounded partition, from which its Coxeter length can be read directly.  It is possible to determine the Coxeter length of the corresponding minimal length coset representative in another way, analogous to the method given in type $\tld{A}_n/A_n$ in Proposition 3.2.8 in \cite{berg-jones-vazirani} by translating Proposition~\ref{p:e_cox} into the language of core partitions.  

If we fill in the boxes of a core along the diagonals with the numbers $1$ through $2n$ instead of the residues, the number at the end of each row corresponds to the runner number of the bead in the abacus which corresponds to that row. This filling allows us to visualize which rows of $\lam$ correspond to beads in the same runner.   Involved in the statement of the proposition below is the determination of the rows $R(i)$ and $r(i)$ of $\lam$ that correspond to the beads $B(i)$ and $b(i)$.

\begin{proposition}\label{p:core_length}
Let $\lam=(\lam_1,\hdots,\lam_k)$ be a symmetric $(2n)$-core partition.  If all parts of $\lam$ are less than or equal to $n$, then $$\ell(\Fw(\lam))=\sum_{1\leq i\leq k} \max(0,\lam_i-i+1+x_0).$$   

Otherwise, define $R(i)$ to be the longest row of $\lam$ to have runner number $i$ or $N-i$ labeling its rightmost box.  Then follow the boundary of $\lam$ extended $n$ steps out in each direction from its center, a path which will involve $n$ vertical steps and $n$ horizontal steps.  Define $r(i)$ to be the unique row of $\lam$ ending with one of the $n$ vertical steps and whose rightmost box is labeled by the same choice of $i$ or $N-i$ as for $R(i)$.  Define $d$ to be the number of rows of $\lam$ whose box on the main diagonal has hook length greater than $2n$.  Then,
\begin{equation*}
\ell\big(\Fw(\lambda)\big) =\sum_{i=1}^{n}
\big(\lambda_{R(i)}-\lam_{r(i)}\big)+(1+x_0+x_i)d+ \sum_{d+1\leq i\leq k} \max(0,\lam_i-i+1+x_0).
\end{equation*}
\end{proposition}
\begin{proof}
In a core of the first type, all boxes are skew. The sum counts the number of boxes in the bounded partition, because the number of boxes between the diagonal and $\lam_i$ inclusive is $\lam_i-i+1$.  

In a core of the second type, there exists a box in $\lam$ with hook length at least $2n$.  Hence the boundary of $\lam$ extends more than $n$ steps out in each direction from its center, and $r(i)$ is well defined since in the abacus $a=\Fa(\lam)$, bead $b(i)$ exists between $n+1$ and $N+n$ on runner $B(i)\p N$, the runner number of row $R(i)$.  The first two terms in the formula translate directly from Proposition~\ref{p:e_cox} and correspond to contributions from beads $b>N+n$.  Beads $b\leq N+n$ correspond to rows of $\lam$ starting with row $d+1$; again the number of boxes between the diagonal and $\lam_i$ inclusive is $\lam_i-i+1$.
\end{proof}

\begin{example}
Consider the core $\lam=(12, 12, 8, 8, 7, 5, 5, 4, 2, 2, 2, 2)\in \tld{B}_3/D_3$, pictured in Figure~\ref{fig:compact} with runner numbers.  

\begin{figure}[h] 
\begin{center}
\epsfig{figure=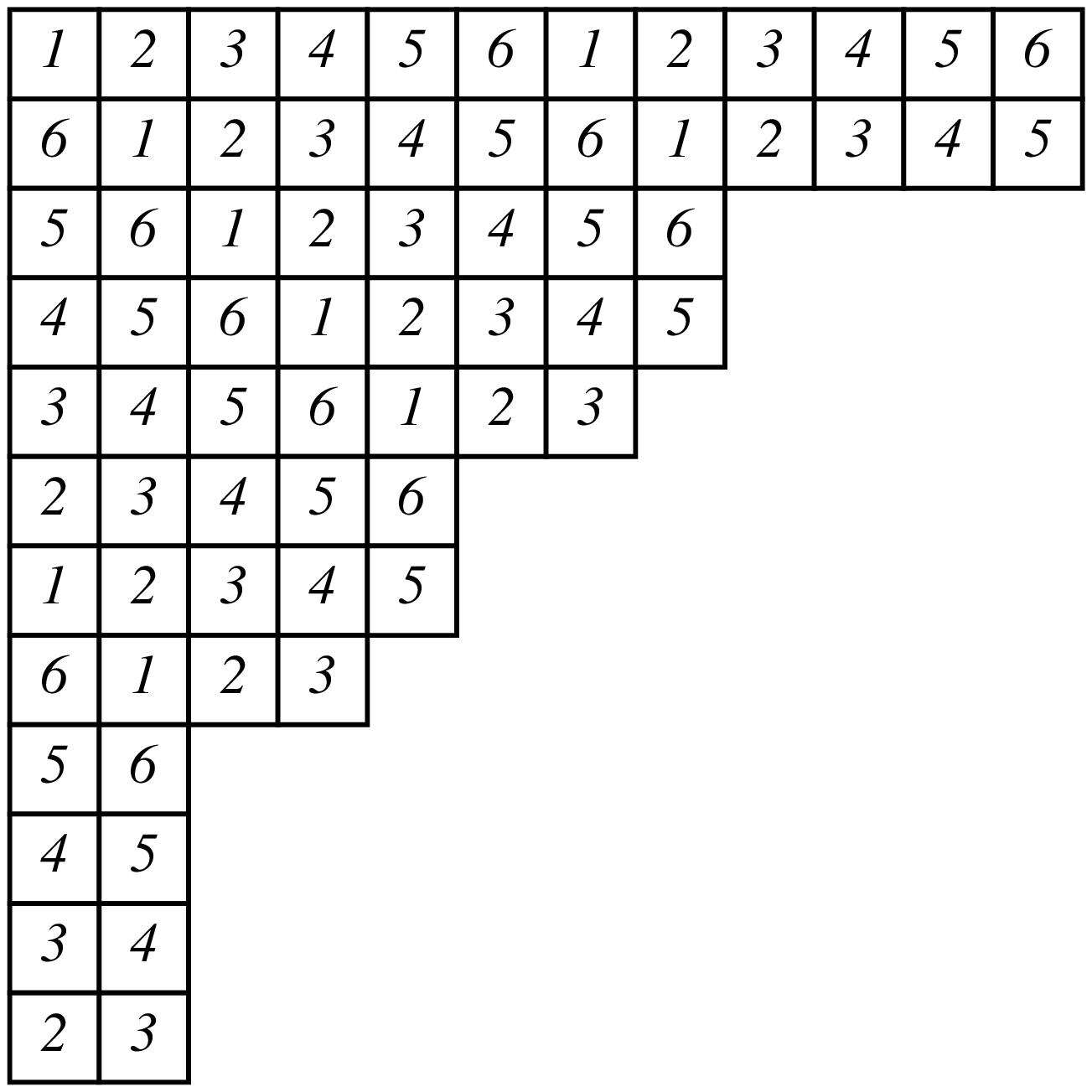,width=2in}
\end{center}
\caption{The core partition $\lam=(12, 12, 8, 8, 7, 5, 5, 4, 2, 2, 2, 2)\in \tld{B}_3/D_3$.  The numbers in the boxes are the runner numbers.}
\label{fig:compact}
\end{figure}

\noindent
This partition has parts of size larger than $3$, so we apply the second half of Proposition~\ref{p:core_length}.  The runners corresponding to $i=1$, $2$, and $3$ are have runner numbers $6$, $5$, and $3$.  Consequently, 
$\lam_{R(1)}=\lam_1=12$, 
$\lam_{R(2)}=\lam_2=12$, 
$\lam_{R(3)}=\lam_5=7$, 
$\lam_{r(1)}=\lam_6=5$, 
$\lam_{r(2)}=\lam_7=5$, 
$\lam_{r(3)}=\lam_5=7$, and $d=4$.
We conclude that $\ell\big(\Fw(\lam)\big)=(12-5)+(12-5)+(7-7)+(1+0-1)\cdot 4+(3)=17$.  
\end{example}

Another formula for Coxeter length can be obtained summing contributions from
the longest rows in $\Fc(w)$ ending with a given runner number, and then
subtracting terms to account for boxes that are counted twice in the peeling
process.  This gives a second analogue of Proposition 3.2.8 in
\cite{berg-jones-vazirani}.

\begin{definition}
Let $\lambda$ be a symmetric $(2n)$-core.  Let $D$ be the lowest box on the main
diagonal of $\lambda$.  For $1 \leq i \leq n$, let $R(i)$ be index of the
longest row $\lambda_{R(i)}$ of $\lambda$ having rightmost box from runner $i$
or $N-i$.

The {\bf rim} of $\lambda$ consists of the boxes from $\lambda$ that have no
box lying directly to the southeast.  We define the {\bf rim walk} $W(i)$ from
the rightmost box $B$ of $\lambda_{R(i)}$ to be the collection of boxes
encountered when walking along the rim of $\lambda$ from $B$ towards $D$, ending
at the last box encountered from runner $\min(i, N-i)$ prior to $D$.  We let
$h(i)$ denote the {\bf height} of $W(i)$, defined to be the number of rows of
$\lambda$ that intersect $W(i)$ and end in a box having a runner number that is
different from the last box of $\lambda_{R(i)}$.
\end{definition}

\begin{theorem}
Let $\lambda$ be a symmetric $(2n)$-core.  The Coxeter length of $\Fw(\lambda)$ is
\[ \ell(\Fw(\lambda)) = \sum_{1 \leq i \leq n} \left(\lambda_{R(i)} - R(i) - h(i) + 1 \right) + x_0 d_0 + x_n d_n. \]
where $d_i$ is the number of boxes on the $i$-th diagonal of $\lambda$.
\end{theorem}
\begin{proof}
By Theorem~\ref{t:upperfromlam}, the Coxeter length is equal to the number of
boxes in the bounded diagram $\widetilde{U}_{\lambda}$ of $\lambda$.  There is
one row in $\widetilde{U}_{\lambda}$ for each bead in $\Fa(\lambda)$ succeeding
$N$ in reading order, and the number of boxes on each row in
$\widetilde{U}_{\lambda}$ is the number of skew boxes in the row, together with
entries from the $0$-th and $n$-th diagonals, depending on the Coxeter type.

We claim that the number of skew boxes from all rows ending with a box from
runner $i$ is equal to $\lambda_{R(i)} - R(i) + 1 - h(i)$.  Depending on the
Coxeter type, we will also subtract boxes corresponding to $x_0 d_0$ and $x_n
d_n$ in the last step of the construction of $\widetilde{U}_{\lambda}$.

To see why the claim is true, consider the last bead $b$ in reading order lying
on runner $i$.  This bead has $\lambda_{R(i)} - R(i) + 1$ gaps lying weakly
between $g(b)$ and $b$.  Each gap lying between $N$ and $b$ corresponds to a
skew box for precisely one bead lying on runner $i$ to the right of $N$ in
reading order.  

Since $N+i$ is the earliest bead on runner $i$ in reading order that corresponds
to a row lying above the main diagonal, we must therefore subtract the gaps
prior to position $\max((N+i)-N,g(N+i)) = \max(i,N-i)$ that are not of the form
$g(b')$ for any bead $b'$ on runner $i$.  

By Lemma~\ref{l:ab_balance}, such gaps correspond to beads lying between
$\min(i, N-i)$ and $b$ that do not lie on runner $i$.  This quantity is equal to
the height of our rim walk beginning at the last box of row $R(i)$, which
corresponds to $B$.

Including contributions for each runner $i$ yields the sum given in the formula.
\end{proof}

\begin{example}
Consider the core $\lam=(12, 12, 8, 8, 7, 5, 5, 4, 2, 2, 2, 2)\in \tld{B}_3/D_3$, pictured in Figure~\ref{fig:compact} with runner numbers.  

We have $R_1 = 1$, $R_2 = 2$ and $R_3 = 5$.  The rim walk from the last box $B$ of
$\lam_1$ consists of all $12$ boxes on the rim lying between $B$ and $D$, and
there are $3$ rows (namely, $2$, $4$, and $5$) that do not end in runner $6$.
Hence, $h(1) = 3$.  The rim walk from the last box $B'$ of $\lam_2$ consists of
the $10$ boxes on the rim lying between $B'$ and the box immediately right of
$D$.  There are $2$ rows (namely, $3$ and $5$) that do not end in runner $5$, so
$h(2) = 2$.  The rim walk from the last box $B''$ of $\lam_5$ consists of the
single box $B''$ because there are no other boxes from runner $3$ lying on the
rim between $B''$ and $D$.  Hence, $h(3) = 0$.

Thus, we compute $\ell(\Fw(\lam))$ as
\[ \left( \left((\lambda_{R(1)} - R(1) + 1) - h(1)\right) +
\left((\lambda_{R(2)} - R(2) + 1) - h(2)\right) + \left((\lambda_{R(3)} - R(3) +
1) - h(3)\right) \right) - d_n \]
\[ = \left( (12 - 3) + (11 - 2) + (3 - 0) \right) - 4 = 17. \]
\end{example}

\bigskip
\section{Proofs}\label{s:proofs}
\subsection{Residues for core partitions}

We now turn to the proof of Theorem~\ref{t:core_action}.

\begin{proof}{\em [of Theorem~\ref{t:core_action}]\ \ }
Given $w \in \widetilde{W}_n/W_n$, suppose that $a = \Fa(w)$ and $\lambda =
\Fc(w)$.  Let the residues be assigned to entries of $\N^2$ near the southeast
boundary of the diagram of $\lambda$ as in Definition~\ref{d:residue}.
Consider the application of a generator $s_i$.  

By Lemma~\ref{l:ab_balance}, the midpoint of the boundary lattice path of
$\lambda$ occurs at entry $N$ in the abacus, and this corresponds to the
outermost corner of the lowest box on the main diagonal in $\lambda$.  
In every type, the boxes on the main diagonal are assigned residue 0.  

Because the assignment of fixed residues is constant along northwest-southeast
diagonals, the assignment of fixed residue to entry $(i,j)$ is the same as the
assignment of fixed residue to entry $(i-u, j+v)$ whenever $u, v \geq 0$ and
$u+v = 2n$.  From this it follows that every bead on a given runner is assigned
the same fixed residue.  
Moreover, we find that all of the beads in runner $j$ correspond to boxes with
residue 
\[ \begin{cases} j-1 & \text{ if } 1 \leq j \leq n+1 \\
    2n-j+1 & \text{ if } n+2 \leq j \leq 2n. \\
\end{cases} \]

We observe that the connected components of boxes with fixed residue are always
single boxes.  A box with fixed residue $i$ is removable if and only if it lies
at the end of its row and column, which therefore occurs if and only if it
corresponds to an active bead on runner $i+1$ or $2n-i+1$ with a gap immediately
preceding it in the reading order of the abacus.  Similarly, a box with fixed
residue $i$ is addable if and only if it corresponds to a gap on runner $i+1$ or
$2n-i+1$ with an active bead immediately preceding it in the reading order of
the abacus.  The action of $s_i$ in type $\widetilde{C}_n$ swaps runners $i$ and
$i+1$ as well as $2n-i$ and $2n-i+1$ which therefore interchanges all of the
$i$-addable and $i$-removable boxes.  

Since the abacus is flush, exchanging runners $i$ with $i+1$ and $2n-i$ with
$2n-i+1$ (mod $N$) either adds some set of boxes with residue $i$ to the diagram
of $\lambda$ in the case that $s_i$ is an ascent, or else removes a set of boxes
with residue $i$ in the case that $s_i$ is a descent.  If $s_i$ is neither an
ascent nor a descent then the levels of the lowest beads in the relevant columns
are the same, so the abacus and corresponding core remain unchanged.  

This proves the result in the case when $s_i$ ($0 \leq i \leq n$) is a
generator of type $\widetilde{C}$.  The generators $s_{n-1}$ and $s_n^D$ apply
to entries in an escalator, which we consider below.  The generators $s_1$ and
$s_0^D$ apply to entries in an descalator, and the argument is entirely similar
so we omit it.

To verify the result for $s_{n-1}$ and $s_n^D$, we consider the action of these
generators on all possible abaci.  It suffices to consider the action on a
single row of the abacus since this translates into a connected segment of the
lattice path boundary of the core partition.  Using the results for type
$\widetilde{C}$ generators that we have already shown above, we observe that
the lattice path always begins on the boundary of a box with residue $n-2$ by
induction.

{\bf Case:}  The abacus row contains a single bead among the positions
$\{n-1,n,n+1,n+2\}$.

Then, the abaci satisfy the following commutative diagram.

\[ 
\begin{tabular}{ccc}
    \abacus { \ci{n-1} &  \nc{n} & \nc{n+1} & \nc{n+2}  } & $\stackrel{s_n^D}{\longrightarrow}$ & \abacus { \nc{n-1} &  \nc{n} & \ci{n+1} & \nc{n+2} }  \\
    $\downarrow^{s_{n-1}}$ & & $\downarrow^{s_{n-1}}$ \\
 \abacus { \nc{n-1} &  \ci{n} & \nc{n+1} & \nc{n+2}  } & $\stackrel{s_n^D}{\longrightarrow}$ & \abacus { \nc{n-1} &  \nc{n} & \nc{n+1} & \ci{n+2} }  \\
\end{tabular}
\]

We translate these entries of the abaci into a segment of the boundary lattice
path.  Observe that Lemma~\ref{l:balance_up_low} implies that we are above the
main diagonal, so we assign residues to boxes in the upper escalator
horizontally.  Then, it is straightforward to verify that the action of $s_{i}$
does add (or remove) all addable (removable, respectively) components with
residue $i$, for $i \in \{n-1, n\}$, as shown below.

\[
\begin{tabular}{ccc}
\myscalebox{0.5}{\begin{tikzpicture}
\draw (0,0) node [rectangle, minimum size=1.5cm, inner sep=0pt, draw, anchor=south west] {$n-2$};
\draw (1.5cm,0) node [rectangle, minimum size=1.5cm, inner sep=0pt, anchor=south west, fill=gray!20] {$\ontop{(n-1)}{n}$};
\draw (3.0cm,0) node [rectangle, minimum size=1.5cm, inner sep=0pt, anchor=south west, fill=gray!20] {$n$};
\draw (4.5cm,0) node [rectangle, minimum size=1.5cm, inner sep=0pt, anchor=south west, fill=gray!20] { };
\draw [line width=3pt] (1.5cm,0) node (test1) [circle, fill=black, inner sep=0pt, minimum size=2pt, draw] { }
 -- ++(90:1.5cm)  node (test2) [circle, fill=black, inner sep=0pt, minimum size=2pt, draw] { }
 -- ++(0:1.5cm)  node (test3) [circle, fill=black, inner sep=0pt, minimum size=2pt, draw] { }
 -- ++(0:1.5cm)  node (test4) [circle, fill=black, inner sep=0pt, minimum size=2pt, draw] { }
 -- ++(0:1.5cm)  node (test5) [circle, fill=black, inner sep=0pt, minimum size=2pt, draw] { };
\end{tikzpicture}}  & \raisebox{.2cm}{$\stackrel{s_n^D}{\longrightarrow}$} &
\myscalebox{0.5}{\begin{tikzpicture}
\draw (0,0) node [rectangle, minimum size=1.5cm, inner sep=0pt, draw, anchor=south west] {$n-2$};
\draw (1.5cm,0) node [rectangle, minimum size=1.5cm, inner sep=0pt, draw, anchor=south west, fill=gray!20] {$n$};
\draw (3.0cm,0) node [rectangle, minimum size=1.5cm, inner sep=0pt, draw, anchor=south west, fill=gray!20] {$n$};
\draw (4.5cm,0) node [rectangle, minimum size=1.5cm, inner sep=0pt, anchor=south west, fill=gray!20] {$n-1$};
\draw [line width=3pt] (1.5cm,0) node (test1) [circle, fill=black, inner sep=0pt, minimum size=2pt, draw] { }
 -- ++(0:1.5cm)  node (test2) [circle, fill=black, inner sep=0pt, minimum size=2pt, draw] { }
 -- ++(0:1.5cm)  node (test3) [circle, fill=black, inner sep=0pt, minimum size=2pt, draw] { }
 -- ++(90:1.5cm)  node (test4) [circle, fill=black, inner sep=0pt, minimum size=2pt, draw] { }
 -- ++(0:1.5cm)  node (test5) [circle, fill=black, inner sep=0pt, minimum size=2pt, draw] { };
\end{tikzpicture}}  \\
\end{tabular}
\]

{\bf Case:}  The abacus row contains a single gap among the positions
$\{n-1,n,n+1,n+2\}$.

Then, the abaci satisfy the following commutative diagram.

\[ \begin{tabular}{ccc}
    \abacus { \ci{n-1} &  \ci{n} & \ci{n+1} & \nc{n+2}  } & $\stackrel{s_n^D}{\longrightarrow}$ & \abacus { \ci{n-1} &  \nc{n} & \ci{n+1} & \ci{n+2} }  \\
    $\downarrow^{s_{n-1}}$ & & $\downarrow^{s_{n-1}}$ \\
 \abacus { \ci{n-1} &  \ci{n} & \nc{n+1} & \ci{n+2}  } & $\stackrel{s_n^D}{\longrightarrow}$ & \abacus { \nc{n-1} &  \ci{n} & \ci{n+1} & \ci{n+2} }  \\
\end{tabular} \]

We translate these entries of the abaci into a segment of the boundary lattice
path.  Observe that Lemma~\ref{l:balance_up_low} implies that we are below the
main diagonal, so we assign residues to boxes in the lower escalator
vertically.  Then, it is straightforward to verify that the action of $s_{i}$
does add (or remove) all addable (removable, respectively) components with
residue $i$, for $i \in \{n-1, n\}$, as shown below.

\[
\begin{tabular}{ccc}
\myscalebox{0.5}{\begin{tikzpicture}
\draw (0,0) node [rectangle, minimum size=1.5cm, inner sep=0pt, draw, anchor=south west] {$n-2$};
\draw (0,1.5cm) node [rectangle, minimum size=1.5cm, inner sep=0pt, draw, anchor=south west, fill=gray!20] { };
\draw (0,3.0cm) node [rectangle, minimum size=1.5cm, inner sep=0pt, draw, anchor=south west, fill=gray!20] { };
\draw (0,4.5cm) node [rectangle, minimum size=1.5cm, inner sep=0pt, draw, anchor=south west, fill=gray!20] { };
\draw (1.5cm,0) node [rectangle, minimum size=1.5cm, inner sep=0pt, anchor=south west, fill=gray!20] { };
\draw (1.5cm,1.5cm) node [rectangle, minimum size=1.5cm, inner sep=0pt, anchor=south west, fill=gray!20] {$n$};
\draw (1.5cm,3.0cm) node [rectangle, minimum size=1.5cm, inner sep=0pt, anchor=south west, fill=gray!20] {$\ontop{(n-1)}{n}$};
\draw (1.5cm,4.5cm) node [rectangle, minimum size=1.5cm, inner sep=0pt, draw, anchor=south west] {$n-2$};
\draw [line width=3pt] (1.5cm,0) node (test1) [circle, fill=black, inner sep=0pt, minimum size=2pt, draw] { }
 -- ++(90:1.5cm)  node (test2) [circle, fill=black, inner sep=0pt, minimum size=2pt, draw] { }
 -- ++(90:1.5cm)  node (test3) [circle, fill=black, inner sep=0pt, minimum size=2pt, draw] { }
 -- ++(90:1.5cm)  node (test4) [circle, fill=black, inner sep=0pt, minimum size=2pt, draw] { }
 -- ++(0:1.5cm)  node (test5) [circle, fill=black, inner sep=0pt, minimum size=2pt, draw] { };
\end{tikzpicture}}  & \raisebox{1.5cm}{$\stackrel{s_n^D}{\longrightarrow}$} &
\myscalebox{0.5}{\begin{tikzpicture}
\draw (0,0) node [rectangle, minimum size=1.5cm, inner sep=0pt, draw, anchor=south west] {$n-2$};
\draw (0,1.5cm) node [rectangle, minimum size=1.5cm, inner sep=0pt, draw, anchor=south west, fill=gray!20] { };
\draw (0,3.0cm) node [rectangle, minimum size=1.5cm, inner sep=0pt, draw, anchor=south west, fill=gray!20] { };
\draw (0,4.5cm) node [rectangle, minimum size=1.5cm, inner sep=0pt, draw, anchor=south west, fill=gray!20] { };
\draw (1.5cm,0) node [rectangle, minimum size=1.5cm, inner sep=0pt, anchor=south west, fill=gray!20] {$n-1$};
\draw (1.5cm,1.5cm) node [rectangle, minimum size=1.5cm, inner sep=0pt, draw, anchor=south west, fill=gray!20] {$n$};
\draw (1.5cm,3.0cm) node [rectangle, minimum size=1.5cm, inner sep=0pt, draw, anchor=south west, fill=gray!20] {$n$};
\draw (1.5cm,4.5cm) node [rectangle, minimum size=1.5cm, inner sep=0pt, draw, anchor=south west] {$n-2$};
\draw [line width=3pt] (1.5cm,0) node (test1) [circle, fill=black, inner sep=0pt, minimum size=2pt, draw] { }
 -- ++(90:1.5cm)  node (test2) [circle, fill=black, inner sep=0pt, minimum size=2pt, draw] { }
 -- ++(0:1.5cm)  node (test3) [circle, fill=black, inner sep=0pt, minimum size=2pt, draw] { }
 -- ++(90:1.5cm)  node (test4) [circle, fill=black, inner sep=0pt, minimum size=2pt, draw] { }
 -- ++(90:1.5cm)  node (test5) [circle, fill=black, inner sep=0pt, minimum size=2pt, draw] { };
\end{tikzpicture}}  \\
\end{tabular}
\]

{\bf Case:}  The abacus row contains two beads in one of the configurations
shown below.

\[ \begin{tabular}{ccc}
\abacus { \ci{n-1} &  \ci{n} & \nc{n+1} & \nc{n+2}  } & $\stackrel{s_n^D}{\longrightarrow}$ & \abacus { \nc{n-1} &  \nc{n} & \ci{n+1} & \ci{n+2} }  \\
\end{tabular} 
\]

We translate these entries of the abaci into a segment of the boundary lattice
path.  

If we are above the main diagonal, then we assign residues to the boxes in the
upper escalator horizontally, as shown below.  

\[
\begin{tabular}{ccc}
\myscalebox{0.5}{\begin{tikzpicture}
\draw (0,0) node [rectangle, minimum size=1.5cm, inner sep=0pt, draw, anchor=south west, fill=gray!20] {$n-1$};
\draw (1.5cm,0) node [rectangle, minimum size=1.5cm, inner sep=0pt, anchor=south west, fill=gray!20] {$n$};
\draw (3.0cm,0) node [rectangle, minimum size=1.5cm, inner sep=0pt, anchor=south west, fill=gray!20] {$n$};
\draw (4.5cm,0) node [rectangle, minimum size=1.5cm, inner sep=0pt, anchor=south west] {$n-2$};
\draw (0,-1.5cm) node [rectangle, minimum size=1.5cm, inner sep=0pt, draw, anchor=south west] {$n-2$};
\draw (1.5cm,-1.5cm) node [rectangle, minimum size=1.5cm, inner sep=0pt, anchor=south west, fill=gray!20] {$\ontop{(n-1)}{n}$};
\draw (3.0cm,-1.5cm) node [rectangle, minimum size=1.5cm, inner sep=0pt, anchor=south west, fill=gray!20] {$n$};
\draw (4.5cm,-1.5cm) node [rectangle, minimum size=1.5cm, inner sep=0pt, anchor=south west, fill=gray!20] { };
\draw [line width=3pt] (1.5cm,-1.5cm) node (test1) [circle, fill=black, inner sep=0pt, minimum size=2pt, draw] { }
 -- ++(90:1.5cm)  node (test2) [circle, fill=black, inner sep=0pt, minimum size=2pt, draw] { }
 -- ++(90:1.5cm)  node (test3) [circle, fill=black, inner sep=0pt, minimum size=2pt, draw] { }
 -- ++(0:1.5cm)  node (test4) [circle, fill=black, inner sep=0pt, minimum size=2pt, draw] { }
 -- ++(0:1.5cm)  node (test5) [circle, fill=black, inner sep=0pt, minimum size=2pt, draw] { };
\end{tikzpicture}}  & \raisebox{0.75cm}{$\stackrel{s_n^D}{\longrightarrow}$} &
\myscalebox{0.5}{\begin{tikzpicture}
\draw (0,0) node [rectangle, minimum size=1.5cm, inner sep=0pt, draw, anchor=south west, fill=gray!20] { };
\draw (1.5cm,0) node [rectangle, minimum size=1.5cm, inner sep=0pt, draw, anchor=south west, fill=gray!20] {$n$};
\draw (3.0cm,0) node [rectangle, minimum size=1.5cm, inner sep=0pt, draw, anchor=south west, fill=gray!20] {$\ontop{(n-1)}{n}$};
\draw (4.5cm,0) node [rectangle, minimum size=1.5cm, inner sep=0pt, anchor=south west] {$n-2$};
\draw (0,-1.5cm) node [rectangle, minimum size=1.5cm, inner sep=0pt, draw, anchor=south west] {$n-2$};
\draw (1.5cm,-1.5cm) node [rectangle, minimum size=1.5cm, inner sep=0pt, draw, anchor=south west, fill=gray!20] {$n$};
\draw (3.0cm,-1.5cm) node [rectangle, minimum size=1.5cm, inner sep=0pt, draw, anchor=south west, fill=gray!20] {$n$};
\draw (4.5cm,-1.5cm) node [rectangle, minimum size=1.5cm, inner sep=0pt, anchor=south west, fill=gray!20] { };
\draw [line width=3pt] (1.5cm,-1.5cm) node (test1) [circle, fill=black, inner sep=0pt, minimum size=2pt, draw] { }
 -- ++(0:1.5cm)  node (test2) [circle, fill=black, inner sep=0pt, minimum size=2pt, draw] { }
 -- ++(0:1.5cm)  node (test3) [circle, fill=black, inner sep=0pt, minimum size=2pt, draw] { }
 -- ++(90:1.5cm)  node (test4) [circle, fill=black, inner sep=0pt, minimum size=2pt, draw] { }
 -- ++(90:1.5cm)  node (test5) [circle, fill=black, inner sep=0pt, minimum size=2pt, draw] { };
\end{tikzpicture}}  \\
\end{tabular}
\]

If we are below the main diagonal, then we assign residues to the boxes in the
lower escalator vertically, as shown below.

\[
\begin{tabular}{ccc}
\myscalebox{0.5}{\begin{tikzpicture}
\draw (0,0) node [rectangle, minimum size=1.5cm, inner sep=0pt, draw, anchor=south west, fill=gray!20] {$\ontop{(n-1)}{n}$};
\draw (1.5cm,0) node [rectangle, minimum size=1.5cm, inner sep=0pt, anchor=south west, fill=gray!20] {$n$};
\draw (3.0cm,0) node [rectangle, minimum size=1.5cm, inner sep=0pt, anchor=south west, fill=gray!20] {$\ontop{(n-1)}{n}$};
\draw (4.5cm,0) node [rectangle, minimum size=1.5cm, inner sep=0pt, anchor=south west] {$n-2$};
\draw (0,-1.5cm) node [rectangle, minimum size=1.5cm, inner sep=0pt, draw, anchor=south west] {$n-2$};
\draw (1.5cm,-1.5cm) node [rectangle, minimum size=1.5cm, inner sep=0pt, anchor=south west, fill=gray!20] {$n$};
\draw (3.0cm,-1.5cm) node [rectangle, minimum size=1.5cm, inner sep=0pt, anchor=south west, fill=gray!20] {$n$};
\draw (4.5cm,-1.5cm) node [rectangle, minimum size=1.5cm, inner sep=0pt, anchor=south west, fill=gray!20] { };
\draw [line width=3pt] (1.5cm,-1.5cm) node (test1) [circle, fill=black, inner sep=0pt, minimum size=2pt, draw] { }
 -- ++(90:1.5cm)  node (test2) [circle, fill=black, inner sep=0pt, minimum size=2pt, draw] { }
 -- ++(90:1.5cm)  node (test3) [circle, fill=black, inner sep=0pt, minimum size=2pt, draw] { }
 -- ++(0:1.5cm)  node (test4) [circle, fill=black, inner sep=0pt, minimum size=2pt, draw] { }
 -- ++(0:1.5cm)  node (test5) [circle, fill=black, inner sep=0pt, minimum size=2pt, draw] { };
\end{tikzpicture}}  & \raisebox{0.75cm}{$\stackrel{s_n^D}{\longrightarrow}$} &
\myscalebox{0.5}{\begin{tikzpicture}
\draw (0,0) node [rectangle, minimum size=1.5cm, inner sep=0pt, draw, anchor=south west, fill=gray!20] {$\ontop{(n-1)}{n}$};
\draw (1.5cm,0) node [rectangle, minimum size=1.5cm, inner sep=0pt, draw, anchor=south west, fill=gray!20] {$n$};
\draw (3.0cm,0) node [rectangle, minimum size=1.5cm, inner sep=0pt, draw, anchor=south west, fill=gray!20] {$n$};
\draw (4.5cm,0) node [rectangle, minimum size=1.5cm, inner sep=0pt, anchor=south west] {$n-2$};
\draw (0,-1.5cm) node [rectangle, minimum size=1.5cm, inner sep=0pt, draw, anchor=south west] {$n-2$};
\draw (1.5cm,-1.5cm) node [rectangle, minimum size=1.5cm, inner sep=0pt, draw, anchor=south west, fill=gray!20] {$\ontop{(n-1)}{n}$};
\draw (3.0cm,-1.5cm) node [rectangle, minimum size=1.5cm, inner sep=0pt, draw, anchor=south west, fill=gray!20] {$n$};
\draw (4.5cm,-1.5cm) node [rectangle, minimum size=1.5cm, inner sep=0pt, anchor=south west, fill=gray!20] { };
\draw [line width=3pt] (1.5cm,-1.5cm) node (test1) [circle, fill=black, inner sep=0pt, minimum size=2pt, draw] { }
 -- ++(0:1.5cm)  node (test2) [circle, fill=black, inner sep=0pt, minimum size=2pt, draw] { }
 -- ++(0:1.5cm)  node (test3) [circle, fill=black, inner sep=0pt, minimum size=2pt, draw] { }
 -- ++(90:1.5cm)  node (test4) [circle, fill=black, inner sep=0pt, minimum size=2pt, draw] { }
 -- ++(90:1.5cm)  node (test5) [circle, fill=black, inner sep=0pt, minimum size=2pt, draw] { };
\end{tikzpicture}}  \\
\end{tabular}
\]

In each subcase, it is straightforward to verify that the action of $s_{i}$ does
add (or remove) all addable (removable, respectively) components with residue $i$,
for $i \in \{n-1, n\}$.

{\bf Case:}  The abacus row contains two beads in one of the configurations
shown below.

\[ \begin{tabular}{ccc}
    \abacus { \ci{n-1} &  \nc{n} & \ci{n+1} & \nc{n+2}  } & $\stackrel{s_{n-1}}{\longrightarrow}$ & \abacus { \nc{n-1} &  \ci{n} & \nc{n+1} & \ci{n+2} }  \\
\end{tabular} 
\]

We translate these entries of the abaci into a segment of the boundary lattice
path.  

If we are above the main diagonal, then we assign residues to the boxes in the
upper escalator horizontally, as shown below.  

\[ \begin{tabular}{ccc}
\myscalebox{0.5}{\begin{tikzpicture}
\draw (0,0) node [rectangle, minimum size=1.5cm, inner sep=0pt, draw, anchor=south west, fill=gray!20] {$n$};
\draw (1.5cm,0) node [rectangle, minimum size=1.5cm, inner sep=0pt, draw, anchor=south west, fill=gray!20] {$n$};
\draw (3.0cm,0) node [rectangle, minimum size=1.5cm, inner sep=0pt, anchor=south west, fill=gray!20] {$n-1$};
\draw (4.5cm,0) node [rectangle, minimum size=1.5cm, inner sep=0pt, anchor=south west] {$n-2$};
\draw (0,-1.5cm) node [rectangle, minimum size=1.5cm, inner sep=0pt, draw, anchor=south west] {$n-2$};
\draw (1.5cm,-1.5cm) node [rectangle, minimum size=1.5cm, inner sep=0pt, anchor=south west, fill=gray!20] {$\ontop{(n-1)}{n}$};
\draw (3.0cm,-1.5cm) node [rectangle, minimum size=1.5cm, inner sep=0pt, anchor=south west, fill=gray!20] {$n$};
\draw (4.5cm,-1.5cm) node [rectangle, minimum size=1.5cm, inner sep=0pt, anchor=south west, fill=gray!20] { };
\draw [line width=3pt] (1.5cm,-1.5cm) node (test1) [circle, fill=black, inner sep=0pt, minimum size=2pt, draw] { }
 -- ++(90:1.5cm)  node (test2) [circle, fill=black, inner sep=0pt, minimum size=2pt, draw] { }
 -- ++(0:1.5cm)  node (test3) [circle, fill=black, inner sep=0pt, minimum size=2pt, draw] { }
 -- ++(90:1.5cm)  node (test4) [circle, fill=black, inner sep=0pt, minimum size=2pt, draw] { }
 -- ++(0:1.5cm)  node (test5) [circle, fill=black, inner sep=0pt, minimum size=2pt, draw] { };
\end{tikzpicture}}  & \raisebox{0.75cm}{$\stackrel{s_{n-1}}{\longrightarrow}$} &
\myscalebox{0.5}{\begin{tikzpicture}
\draw (0,0) node [rectangle, minimum size=1.5cm, inner sep=0pt, draw, anchor=south west, fill=gray!20] { };
\draw (1.5cm,0) node [rectangle, minimum size=1.5cm, inner sep=0pt, draw, anchor=south west, fill=gray!20] {$n$};
\draw (3.0cm,0) node [rectangle, minimum size=1.5cm, inner sep=0pt, draw, anchor=south west, fill=gray!20] {$\ontop{(n-1)}{n}$};
\draw (4.5cm,0) node [rectangle, minimum size=1.5cm, inner sep=0pt, anchor=south west] {$n-2$};
\draw (0,-1.5cm) node [rectangle, minimum size=1.5cm, inner sep=0pt, draw, anchor=south west] {$n-2$};
\draw (1.5cm,-1.5cm) node [rectangle, minimum size=1.5cm, inner sep=0pt, draw, anchor=south west, fill=gray!20] {$n-1$};
\draw (3.0cm,-1.5cm) node [rectangle, minimum size=1.5cm, inner sep=0pt, anchor=south west, fill=gray!20] {$n$};
\draw (4.5cm,-1.5cm) node [rectangle, minimum size=1.5cm, inner sep=0pt, anchor=south west, fill=gray!20] {$n$};
\draw [line width=3pt] (1.5cm,-1.5cm) node (test1) [circle, fill=black, inner sep=0pt, minimum size=2pt, draw] { }
 -- ++(0:1.5cm)  node (test2) [circle, fill=black, inner sep=0pt, minimum size=2pt, draw] { }
 -- ++(90:1.5cm)  node (test3) [circle, fill=black, inner sep=0pt, minimum size=2pt, draw] { }
 -- ++(0:1.5cm)  node (test4) [circle, fill=black, inner sep=0pt, minimum size=2pt, draw] { }
 -- ++(90:1.5cm)  node (test5) [circle, fill=black, inner sep=0pt, minimum size=2pt, draw] { };
\end{tikzpicture}}  \\
\end{tabular}
\]

If we are below the main diagonal, then we assign residues to the boxes in the
lower escalator vertically, as shown below.

\[ \begin{tabular}{ccc}
\myscalebox{0.5}{\begin{tikzpicture}
\draw (0,0) node [rectangle, minimum size=1.5cm, inner sep=0pt, draw, anchor=south west, fill=gray!20] {$\ontop{(n-1)}{n}$};
\draw (1.5cm,0) node [rectangle, minimum size=1.5cm, inner sep=0pt, draw, anchor=south west, fill=gray!20] {$n$};
\draw (3.0cm,0) node [rectangle, minimum size=1.5cm, inner sep=0pt, anchor=south west, fill=gray!20] {$\ontop{(n-1)}{n}$};
\draw (4.5cm,0) node [rectangle, minimum size=1.5cm, inner sep=0pt, anchor=south west] {$n-2$};
\draw (0,-1.5cm) node [rectangle, minimum size=1.5cm, inner sep=0pt, draw, anchor=south west] {$n-2$};
\draw (1.5cm,-1.5cm) node [rectangle, minimum size=1.5cm, inner sep=0pt, anchor=south west, fill=gray!20] {$n-1$};
\draw (3.0cm,-1.5cm) node [rectangle, minimum size=1.5cm, inner sep=0pt, anchor=south west, fill=gray!20] {$n$};
\draw (4.5cm,-1.5cm) node [rectangle, minimum size=1.5cm, inner sep=0pt, anchor=south west, fill=gray!20] { };
\draw [line width=3pt] (1.5cm,-1.5cm) node (test1) [circle, fill=black, inner sep=0pt, minimum size=2pt, draw] { }
 -- ++(90:1.5cm)  node (test2) [circle, fill=black, inner sep=0pt, minimum size=2pt, draw] { }
 -- ++(0:1.5cm)  node (test3) [circle, fill=black, inner sep=0pt, minimum size=2pt, draw] { }
 -- ++(90:1.5cm)  node (test4) [circle, fill=black, inner sep=0pt, minimum size=2pt, draw] { }
 -- ++(0:1.5cm)  node (test5) [circle, fill=black, inner sep=0pt, minimum size=2pt, draw] { };
\end{tikzpicture}}  & \raisebox{0.75cm}{$\stackrel{s_{n-1}}{\longrightarrow}$} &
\myscalebox{0.5}{\begin{tikzpicture}
\draw (0,0) node [rectangle, minimum size=1.5cm, inner sep=0pt, draw, anchor=south west, fill=gray!20] {$\ontop{(n-1)}{n}$};
\draw (1.5cm,0) node [rectangle, minimum size=1.5cm, inner sep=0pt, draw, anchor=south west, fill=gray!20] {$n$};
\draw (3.0cm,0) node [rectangle, minimum size=1.5cm, inner sep=0pt, draw, anchor=south west, fill=gray!20] {$n-1$};
\draw (4.5cm,0) node [rectangle, minimum size=1.5cm, inner sep=0pt, anchor=south west] {$n-2$};
\draw (0,-1.5cm) node [rectangle, minimum size=1.5cm, inner sep=0pt, draw, anchor=south west] {$n-2$};
\draw (1.5cm,-1.5cm) node [rectangle, minimum size=1.5cm, inner sep=0pt, draw, anchor=south west, fill=gray!20] {$\ontop{(n-1)}{n}$};
\draw (3.0cm,-1.5cm) node [rectangle, minimum size=1.5cm, inner sep=0pt, anchor=south west, fill=gray!20] {$n$};
\draw (4.5cm,-1.5cm) node [rectangle, minimum size=1.5cm, inner sep=0pt, anchor=south west, fill=gray!20] { };
\draw [line width=3pt] (1.5cm,-1.5cm) node (test1) [circle, fill=black, inner sep=0pt, minimum size=2pt, draw] { }
 -- ++(0:1.5cm)  node (test2) [circle, fill=black, inner sep=0pt, minimum size=2pt, draw] { }
 -- ++(90:1.5cm)  node (test3) [circle, fill=black, inner sep=0pt, minimum size=2pt, draw] { }
 -- ++(0:1.5cm)  node (test4) [circle, fill=black, inner sep=0pt, minimum size=2pt, draw] { }
 -- ++(90:1.5cm)  node (test5) [circle, fill=black, inner sep=0pt, minimum size=2pt, draw] { };
\end{tikzpicture}}  \\
\end{tabular}
\]

In each subcase, it is straightforward to verify that the action of $s_{i}$ does
add (or remove) all addable (removable, respectively) components with residue $i$,
for $i \in \{n-1, n\}$.

{\bf Case:}  The abacus positions $\{n-1,n,n+1,n+2\}$ are all beads or all gaps.

In this case, both $s_{n-1}$ and $s_n^D$ fix the boundary lattice path
segment.  It is straightforward to see that the corresponding core partitions
have no addable nor removable $i$-boxes for $i \in \{n-1, n\}$.

Observe that Lemma~\ref{l:balance_up_low} prohibits abaci in either of the
configurations shown below.

\[ \begin{tabular}{ccc}
    \abacus { \ci{n-1} &  \nc{n} & \nc{n+1} & \ci{n+2}  } & & \abacus { \nc{n-1} &  \ci{n} & \ci{n+1} & \nc{n+2} }  \\
\end{tabular} 
\]

This exhausts the cases.
\end{proof}

\bigskip
\subsection{The upper partition}\label{s:upper_partition}

We now turn to the proof of Theorem~\ref{t:upperfromlam}.

\begin{proof}{\em [of Theorem~\ref{t:upperfromlam}]\ \ }
Fix a core partition $\lambda$ and its associated abacus $a = \Fa(\lambda)$.  Recall that a box in $\lambda$ having hook length $< 2n$ is called {\bf skew}.  Define the bounded diagram $\widetilde{U}_\lam$ as in Section~\ref{s:bounded_diagram}.
We say that a box of $\lambda$ lying in the bounded diagram $\widetilde{U}_\lambda$ is {\bf bounded}.  We will prove that the bounded boxes are the boxes peeled in the central peeling procedure. 

To prove Theorem~\ref{t:upperfromlam}, we work by induction on Coxeter length, investigating the application of one step of the central peeling process.  For the remainder of this proof, let $B$ be the rightmost box in the row $r$ of $\lam$ containing the lowest box on the reference diagonal, $b$ be its corresponding active bead in $\Fa(\lam)$, and let $i$ be the residue of $B$.  
When we apply $s_i$ to remove $B$, we claim that:
\begin{itemize}
    \item[(1)]  $B$ is bounded.
    \item[(2)]  We remove no other bounded box.
    \item[(3)]  If a box $B'$ was bounded in $\lambda$, then $B'$ is also bounded in $s_i(\lambda)$.
\end{itemize}
Proving these claims complete the proof of Theorem~\ref{t:upperfromlam} because they show that during the central peeling procedure the bounded boxes of $\lam$ remain bounded boxes in intermediary steps and that exactly one bounded box of $U_\lam$ is peeled in each step.

\medskip
In terms of the abacus, $b$ is essentially the first bead succeeding $N$ in reading order.  However, if we are in a type that uses $s_0^D$ then there may be a bead in position $N+1$ that cannot be moved to the left because position $N+2$ is a gap.  In this case, $b$ is the next bead in reading order after $N+1$.  In all cases, $b \leq 2N+1$ by the Balance Lemma~\ref{l:ol_balance}.

\medskip
We first prove (1) using the definition of $(2n)$-core.  When we are in a type that uses $s_0^C$, then $r$ is the row containing the lowest box on the main diagonal.  All boxes in row $r$ starting from the first diagonal to $B$ all have hook length less than $2n$, so they are all skew and consequently bounded as well, including $B$.  When we are in a type that uses $s_0^D$, then $r$ is the row containing the lowest box on the first diagonal.  All boxes in row $r$ starting from the first diagonal to $B$ all have hook length less than or equal to $2n$; equality occurs only when $r$ has a box in the $2n$-th diagonal and row $r+1$ has a box on the main diagonal.  The skew boxes in $r$ are therefore all boxes between the second and $2n$-th diagonal; this implies that the bounded boxes in $r$ are all boxes between the main and $(2n-1)$-st diagonal, this last box being $B$ by Definition~\ref{def:cpp}.

\medskip
To prove (2), consider some box $B'$ that is removed when applying $s_i$.  We will show that there are non-skew boxes in the row containing $B'$ which imply that $B'$ is not bounded.   Define $\widehat{B}$ to be the lowest box on the $r$-th column.  This is the reflection of $B$ by the main diagonal and in the abacus this column corresponds to the symmetric gap $g(b)$.  Suppose that we apply generator $s_i$ for $1\leq i\leq n-1$ or $s_i=s_n^C$.   We note that both $\widehat{B}$ and $B'$ have residue $i$ and are in non-consecutive diagonals containing this residue.  Therefore the box that is at the same time above $\widehat{B}$ and to the left of $B'$ has hook length at least $2n$ and is therefore not skew.  

When we apply the generator $s_0^C$, the box $B$ is on the main diagonal and the box $B'$ is on the $j$-th diagonal for $j\geq 2n$.  Therefore the box that is at the same time above $B$ and to the left of $B'$ has hook length at least $2n$ and is therefore not skew. 

When we apply the generator $s_0^D$, we remove two boxes in $B'$'s row; we prove the existence of the two required gaps.  In this case, $B$ is on the first diagonal, with a box directly below, a box to the left, and the box $\widehat{B}$ to the lower-left, all four of which are removed when $s_0^D$ is applied.  The two boxes to the left of $B'$ that are above $B$ and $\widehat{B}$ both have hook length at least $2n$.  

When we apply the generator $s_n^D$, we also remove two boxes in $B'$'s row.  If $B'$ is on the $j$-th diagonal where $j\geq N+n-1$, then the box above $\widehat{B}$ and to the left of $B'$ has hook length at least $2n$, as does its right neighbor.  The last case is that $B'$ is in diagonal $n+1$ while $B$ is in diagonal $n-1$.  In this case, $B$ and $B'$ are in the same block of four boxes which are removed upon the application of $s_n^D$.  There are $n$ bounded boxes in $B'$'s row, not including $B'$.
This exhausts the cases.

\medskip
To prove (3), we show that the number of skew boxes on each row of $\lambda$ above $B$ is equal to the number of skew boxes on each row of $s_i(\lambda)$ above $B$.  Left-justifying these skew boxes to the main diagonal then gives the same bounded region.  

We consider first generators $s_i$ where $1\leq i\leq n-1$ and $s_n^C$.  Let $b'$ be a bead lying to the right of $b$ in reading order.  
If $b'\leq N+n$, then all boxes between the main diagonal and $B'$ are skew.  
No boxes in this row are removed upon application of $s_i$ because the residue of $B'$ is different from the residue of $B$.  
If $b'>N+n$, then the number of skew boxes is the number of gaps between $b'$ and $b'-N$.  
If the action of $s_i$ fixes runner $b'\p N$, then the number of gaps between $b'$ and $b'-N$ does not change.  

Otherwise $b'$ is on runner $b\p N$ or runner $(N-b+1)\p N$.  If $b'$ is on runner $b\p N$, then since $b-1$ is a gap, then so is $b'-N-1$.  If $b'$ is on runner $(N-b+1)\p N$, then we notice that because $b$ is a bead then $N-b$ is a gap, and consequently so is $b'-N-1$.  In both cases, if we restrict our attention to the entries between $b'$ and $b'-N$, we see that in the application of $s_i$, we lose a gap in position $b'-1$ but gain a gap in position $b'-N+1$, so there is no net change.

In the case of $s_0^C$, both $b'$ and $b$ are on runner $1$, so because $b-2=2n$ is a gap, so is $b'-N-2$, and we see that in the application of $s_0^C$, we lose a gap in position $b'-2$ but gain a gap in position $b'-N+2$.

If $b'$ is involved in a transposition under $s_0^D$, then there are two cases.  Either $b$ is in position $N+2$ or $b$ is in position $2N+1$.  In the former case, there is a bead in position $N+1$, and there is a gaps in positions $2n-1$ and $2n$.  Consequently, $b'$ is in either runner $1$ or $2$, and there are gaps $g'$ and $g''$ to the left of $b'$ on runners $2n-1$ and $2n$, as well as on the level below.  When we apply $s_0^D$, we therefore lose the gaps in runners $2n-1$ and $2n$, but then gain them back in runners $1$ and $2$.  In the latter case, the only beads lower than $N$ are in runner $1$, and there is a gap in position $2n$.  Consequently, $b'$ is also in runner $1$, and there are gaps in positions $b'-2$ and $b'-N-2$.  When we apply $s_0^D$, we lose the gap $b'-2$ but then gain back $b'-N+2$.

Finally, if $b'$ is involved in a transposition under $s_n^D$, then $b$ is in position $N+n+1$ or $N+n+2$, and there are gaps in positions $N+n-1$ and $N+n$.  Consequently, $b'$ is in either runner $n+1$ or $n+2$, and there are gaps $g'$ and $g''$ to the left of $b'$ on runners $n$ and $n-1$, as well as on the level below.  When we apply $s_n^D$, we therefore lose the gaps in runners $n$ and $n-1$, but then gain them back in positions $g'-N+2$ and $g''-N+2$.  This completes the proof of (3).
\end{proof}

\bigskip
\section{Future work}\label{s:future_work}

There are various results (\cite{clark-ehrenborg}, \cite{armstrong--rhoades}, \cite{fishel--vazirani},
\cite{hanusa-jones}, to name some recent examples) involving type $\widetilde{A}$
objects that we expect have analogues in the other types.

Another family of combinatorial objects in bijection with $\widetilde{C}_n/C_n$
is the lecture hall partitions of Bousquet-M\'elou and Eriksson \cite{BME}.  It
may be interesting to see whether some analogue of these partitions exists in
the other affine types.

\bigskip
\section*{Acknowledgements}

This work has benefited from conversations with M. Vazirani and C. Berg.  We
also thank D. Armstrong, S. Fishel, R. Green, E. O'Shea, and A. Schilling for
helpful discussions.  C.\ R.\ H.\ Hanusa gratefully acknowledges support from
PSC-CUNY Research Awards PSCREG-41-303 and TRADA-42-115.  We thank the
anonymous reviewer for several helpful suggestions.


\def\cprime{$'$}

\end{document}